\documentclass[a4paper,11pt,twoside]{article}

\usepackage{geometry}
\usepackage[utf8]{inputenc}
\usepackage{lmodern}
\usepackage[T1]{fontenc}
\usepackage{dsfont}
\usepackage{tikz}
\usepackage{amsmath,amssymb,amsthm,empheq,cases}
\usepackage{graphicx}
\usepackage{url}
\usepackage{hyperref}
\hypersetup{colorlinks,
            citecolor=red, 
            filecolor=black,
            linkcolor=blue,
            urlcolor=blue}
\def \dis {\displaystyle}
\def \NN {\mathbb N}

\def \RR {\mathbb R}
\def \CC {\mathbb C}

\def \A {\mathcal{A}}
\def \B {\mathcal{B}}

\def \E {\mathcal{E}}

\def \H {\mathcal{H}}

\def \L {\mathcal{L}}
\def \M {\mathcal{M}}

\def \R {\mathcal{R}}

\def \ecart {\noalign{\medskip}}
\theoremstyle{definition}
\newtheorem{Th}{Theorem}[section]
\newtheorem{Prop}[Th]{Proposition}
\newtheorem{Lem}[Th]{Lemma}
\newtheorem{Cor}[Th]{Corollary}
\newtheorem{Def}[Th]{Definition}
\newtheorem{Rem}[Th]{Remark}

\def \refs #1{section~\ref{#1}}

\def \refD #1{Definition~\ref{#1}}
\def \refT #1{Theorem~\ref{#1}}
\def \refL #1{Lemma~\ref{#1}}
\def \refC #1{Corollary~\ref{#1}}
\def \refP #1{Proposition~\ref{#1}}
\def \refR #1{Remark~\ref{#1}}

\title{Generation of analytic semigroups for some generalized diffusion operators in $L^p$-spaces}

\author{Rabah Labbas, St\'ephane Maingot \& Alexandre Thorel \\ \ecart
{\scriptsize Normandie Univ, UNIHAVRE, LMAH, FR-CNRS-3335, ISCN, 76600 Le Havre, France.}\\
{\scriptsize rabah.labbas@univ-lehavre.fr, stephane.maingot@univ-lehavre.fr, alexandre.thorel@univ-lehavre.fr}}
\date{}

\begin{document}
\maketitle
\begin{abstract}
We consider some generalized diffusion operators of fourth order and their corresponding abstract Cauchy problem. Then, using semigroups techniques and functional calculus, we study the invertibility and the spectral properties of each operator. Therefore, we prove that we have generation of $C_0$-semigroup in each case. We also point out when these semigroups become analytic. 
\end{abstract} 

\section{Introduction}

In this article, we study five abstract problems modelling a class of generalized diffusion problems set on a cylindrical space $\Omega = (a,b)\times \omega$ where $\omega$ is a bounded regular open set of $\RR^{n-1}$, $n\geqslant 2$. By generalized diffusion problems, we mean a linear combination of the laplacian and the biharmonic operator. Here, the biharmonic term represents the long range diffusion, whereas the laplacian represents the short range diffusion. 

This kind of problem arises in various concrete applications in physics, engineering and biology. For instance, in elasticity problems, we can cite \cite{elasticity 1}, \cite{elasticity 2} or \cite{IJPAM}. In electrostatic, we refer to \cite{CV 1}, \cite{electrostatic 1} or \cite{electrostatic 2} and in plates theory, we refer to \cite{hassan}, \cite{plaque 1} or \cite{plaque 2}. In population dynamics, we also refer to \cite{cohen-murray}, \cite{LMMT}, \cite{LLMT}, \cite{okubo} or \cite{ochoa} and references therein cited.

Let $T>0$, $k \in \mathbb{R}$ and $f\in L^p((0,T)\times\Omega)$, $p\in (1,+\infty)$. We consider, as an application model, the following generalized diffusion problem : 
\begin{equation}\label{Pedp gen}
\left\{\hspace{-0.1cm}\begin{array}{llll}
\dis \frac{\partial v}{\partial t} (t,x,y) = -\Delta^2 v(t,x,y) + k \Delta v(t,x,y) + f(t,x,y), & t \in (0,T],~x \in (a,b), ~y \in \omega, \\ \ecart
v(0,x,y) = v_0(x,y), & x \in (a,b),~y \in \omega, \\
v(t,x,\zeta) = \Delta v (t,x,\zeta) = 0, & t \in (0,T],~(x,\zeta) \in (a,b) \times \partial\omega \\
\text{Boundary Conditions (BC)} & \text{on } \{a,b\}\times \omega,
\end{array}\right.
\end{equation}
where $v$ is a density, $v_0$ is given in a suitable space and the boundary conditions (BC) denote one of the following homogeneous boundary conditions:
\begin{equation*}
\left\{\begin{array}{rclrclc}
v(t,a,y)   & = & 0, & v(t,b,y)   & = & 0, & t \in (0,T],~y \in \omega, \\
\partial^2_{x} v(t,a,y) & = & 0, & \partial^2_{x} v(t,b,y) & = & 0, & t \in (0,T],~y \in \omega,
\end{array}\right.
\end{equation*}
\begin{equation*}
\left\{\begin{array}{rclrclc}
\partial_x v(t,a,y) & = & 0, & \partial_x v(t,b,y) & = & 0, & t \in (0,T],~y \in \omega, \\
\partial^2_x v(t,a,y) + \Delta_{y} v(t,a,y) &= & 0, & \partial^2_x v(t,b,y) + \Delta_{y} v(t,b,y) &= & 0, & t \in (0,T],~y \in \omega,
\end{array}\right.
\end{equation*}
\begin{equation*}
\left\{\begin{array}{rclrclc}
v(t,a,y) & = & 0, & v(t,b,y) & = & 0, &t \in (0,T],~ y \in \omega, \\
\partial_x v(t,a,y) & = & 0, & \partial_x v(t,b,y) & = & 0, &t \in (0,T],~ y \in \omega,
\end{array}\right.
\end{equation*}
\begin{equation*}
\left\{\begin{array}{rclrclc}
\partial_x v(t,a,y) & = & 0, & \partial_x v(t,b,y) & = & 0, & t \in (0,T],~y \in \omega, \\
\partial^2_{x} v(t,a,y) & = & 0, & \partial^2_{x} v(t,b,y) & = & 0, & t \in (0,T],~ y \in \omega,
\end{array}\right.
\end{equation*}
or
\begin{equation*}
\left\{\begin{array}{rclrclc} 
v(t,a,y) & = & 0, & v(t,b,y) & = & 0, & t \in (0,T],~y \in \omega, \\
\partial^2_x v(t,a,y) + \Delta_{y} v(t,a,y) & = & 0, & \partial^2_x v(b,y) + \Delta_{y} v(t,b,y) & = & 0, & t \in (0,T],~ y \in \omega.
\end{array}\right.
\end{equation*}
We set
\begin{equation*}
\left\{\begin{array}{l}
D(A_0) := W^{2,p}(\omega) \cap W^{1,p}_0(\omega) \\ \ecart
\forall \psi \in D(A_0), \quad A_0 \psi = \Delta_y \psi.
\end{array}\right.
\end{equation*}
Now, for $i=1,2,3,4,5$, let us introduce the following linear operators which correspond to the abstract formulation of the spatial operator in \eqref{Pedp gen}:
$$\left\{\begin{array}{ccl}
D(\A_{0,i}) & = & \{u \in W^{4,p}(a,b;L^p(\omega))\cap L^p(a,b;D(A_0^2))\text{ and } u'' \in L^p(a,b;D(A_0)) : \text{(BCi)}_0\} \\ \ecart
\left[\A_{0,i} u\right](x) & = & - u^{(4)}(x) - (2A_0 - kI) u''(x) - (A^2_0 - k A_0) u(x), \quad u \in D(\A_{0,i}),~x\in (a,b).
\end{array}\right.$$
Here, (BCi)$_0$, $i=1,2,3,4,5$, represents the following boundary conditions (BCi) in the homogeneous case:
\begin{equation}\tag{BC1}\label{condbord1}
\left\{\begin{array}{rclrcl}
u(a) &= & \varphi_1, & u(b) &= &\varphi_2, \\
u''(a) &= &\varphi_3,& u''(b) &= &\varphi_4,
\end{array}\right.
\end{equation}
\begin{equation}\tag{BC2}\label{condbord2}
\left\{\begin{array}{rclrcl}
u'(a) &= & \varphi_1, & u'(b) &= &\varphi_2, \\
u''(a)+A_0 u(a) &= &\varphi_3, & u''(b)+ A_0 u(b) &= &\varphi_4,
\end{array}\right.
\end{equation}
\begin{equation}\tag{BC3}\label{condbord3}
\left\{\begin{array}{rclrcl}
u(a) &= & \varphi_1, & u(b) &= &\varphi_2, \\
u'(a) &= &\varphi_3, & u'(b) &= &\varphi_4,
\end{array}\right.
\end{equation}
\begin{equation}\tag{BC4}\label{condbord4}
\left\{\begin{array}{rclrcl}
u'(a) &= & \varphi_1, & u'(b) &= &\varphi_2, \\
u''(a) &= &\varphi_3, & u''(b) &= &\varphi_4,
\end{array}\right.
\end{equation}
or
\begin{equation}\tag{BC5}\label{condbord5}
\left\{\begin{array}{rclrcl}
u(a) &= & \varphi_1, & u(b) &= &\varphi_2, \\
u''(a) + A_0 u(a) &= &\varphi_3, & u''(b) + A_0 u(b) &= &\varphi_4,
\end{array}\right.
\end{equation}
where $\varphi_1,\varphi_2,\varphi_3,\varphi_4\in L^p(\omega)$.

In this work, we study operators $\A_i$, $i=1,2,3,4,5$, wherein we have considered a more general operator $A$, satisfying some elliptic assumptions described in \refs{sect linear parabolic case}, instead of $A_0$. Then, we study the spectral equation
$$\left(-\A_i - \lambda I \right)u = g,$$
where $g \in L^p(a,b;X)$ with $p \in (1, +\infty)$ and $X$ a complex Banach space. This leads us to solve the abstract Cauchy problem
\begin{equation}\label{P parabo}
\left\{\begin{array}{l}
\dis v'(t) - \A_i v(t) = f(t), \quad t\in (0,T] \\ \ecart
v(0) = v_0,
\end{array}\right.
\end{equation}
in two cases:
\begin{enumerate}
\item $f : [0,T] \longrightarrow L^p(a,b;X)$ with $p \in (1, +\infty)$ and $v_0$ is in suitable spaces,

\item $f : [0,T] \longrightarrow C^\theta([a,b];X)$ with $\theta \in (0,1)$, $v_0 \in D(\A_i)$ and $f$, $v_0$ satisfying some compatibility condition which will be specified in \refs{sect Applications}.

\end{enumerate}

Moreover, we will study, among others, the optimal regularity of the two following functions
\begin{equation*}
u_\psi (x) = \left( e^{(x-a) \Lambda_2} - e^{(x-a) \Lambda_1}\right) \psi \quad \text{and} \quad v_\psi (x) = \left( e^{(b-x) \Lambda_2} - e^{(b-x) \Lambda_1}\right) \psi,
\end{equation*}
where $\psi \in X$, $x \in (a,b) \subset \RR$, with $a<b$, $\Lambda_2 = \Lambda_1 + \B$; here $\Lambda_1$ is a boundedly invertible operator satisfying the Maximal Regularity property $(\M\R)$, see \refs{section reg diff semi-groupe}, $\B \in \L(X)$ and $\Lambda_1 \B = \B \Lambda_1\text{ on }D(\Lambda_1)$.

This optimal regularity, which is an interesting result in itself, will be very useful for the study of the spectral properties of $\A_i$.

This article is organized as follows. In \refs{sect Def}, we recall some classical definitions and results about sectorial operators and interpolation spaces. In \refs{sect linear steady case}, we study a class of general fourth order linear abstract problem where we show the existence and uniqueness of the classical solution for each problem. In \refs{sect linear parabolic case}, using the results of the previous section, we study the spectral properties of each $\A_i$ for $i=1,2,3,4,5$ and we prove that $-\A_i$ is sectorial. We then obtain that $\A_i$ generates a $C_0$-semigroup which becomes analytic for some $A$ as $A_0$. Finally, \refs{sect Applications} is devoted to an application focused on the study of problem \eqref{Pedp gen}. 

\section{Definitions and prerequisites} \label{sect Def}

\subsection{The class of Bounded Imaginary Powers of operators}

\begin{Def}\label{Def Banach}
A Banach space $X$ is a UMD space if and only if for all $p\in(1,+\infty)$, the Hilbert transform is bounded from $L^p(\RR,X)$ into itself (see \cite{bourgain} and \cite{burkholder}).
\end{Def}
\begin{Def}\label{Def op Sect}
Let $\alpha \in(0,\pi)$. Sect($\alpha$) denotes the space of closed linear operators $T_1$ which satisfying
$$
\begin{array}{l}
i)\quad \sigma(T_1)\subset \overline{S_{\alpha}},\\ \ecart
ii)\quad\forall~\alpha'\in (\alpha,\pi),\quad \sup\left\{\|\lambda(\lambda\,I-T_1)^{-1}\|_{\L(X)} : ~\lambda\in\CC\setminus\overline{S_{\alpha'}}\right\}<+\infty,
\end{array}
$$
where
\begin{equation}\label{defsector}
S_\alpha\;:=\;\left\{\begin{array}{lll}
\left\{ z \in \CC : z \neq 0 ~~\text{and}~~ |\arg(z)| < \alpha \right\} & \text{if} & \alpha \in (0, \pi], \\ \ecart
\,(0,+\infty) & \text{if} & \alpha = 0,
\end{array}\right.
\end{equation} 
see \cite{haase}, p. 19. Such an operator $T_1$ is called sectorial operator of angle $\alpha$.
\end{Def}
\begin{Rem}
From \cite{komatsu}, p. 342, we know that any injective sectorial operator~$T_1$ admits imaginary powers $T_1^{is}$, $s\in\RR$, but, in general, $T_1^{is}$ is not bounded.
\end{Rem} 
\begin{Def}
Let $\theta \in [0, \pi)$. We denote by BIP$(X,\theta)$, the class of sectorial injective operators $T_2$ such that
\begin{itemize}
\item[] $i) \quad ~~\overline{D(T_2)} = \overline{R(T_2)} = X,$

\item[] $ii) \quad ~\forall~ s \in \RR, \quad T_2^{is} \in \L(X),$

\item[] $iii) \quad \exists~ C \geq 1 ,~ \forall~ s \in \RR, \quad ||T_2^{is}||_{\L(X)} \leq C e^{|s|\theta}$,
\end{itemize}
see~\cite{pruss-sohr}, p. 430.
\end{Def}

\subsection{Interpolation spaces}\label{sect interp}

Here we recall some properties about real interpolation spaces in particular cases.
\begin{Def}
Let $T_3 : D(T_3) \subset X \longrightarrow X$ be a linear operator such that
\begin{equation}\label{def T3}
(0,+\infty) \subset \rho(T_3)\quad \text{and}\quad \exists~C>0:\forall~t>0, \quad \|t(T_3-tI)^{{\mbox{-}}1}\|_{\L(X)} \leqslant C.
\end{equation}
Let $m \in \NN \setminus \{0\}$, $\theta \in (0,1)$ and $q \in [1,+\infty]$. We will use the real interpolation spaces 
$$(D(T_3^m),X)_{\theta,q} = (X,D(T_3^m))_{1{\mbox{-}}\theta,q},$$ 
defined, for instance, \cite{lions-peetre} or \cite{lunardi}. 

In particular, for $m=1$, we have the following characterization
$$(D(T_3),X)_{\theta,q} := \left\{\psi \in X:t\longmapsto t^{1-\theta} \|T_3(T_3-tI)^{-1}\psi\|_X \in L_*^q(0,+\infty)\right\},$$
where $L_*^q(0,+\infty)$ is given by
$$L_*^q(0,+\infty;\CC) := \left\{f \in L^q(0,+\infty) : \left(\int_0^{+\infty} |f(t)|^q\frac{dt}{t}\right)^{1/q} < + \infty \right\}, \quad \text{for } q \in [1,+\infty),$$
and for $q = +\infty$, by
$$L_*^\infty(0,+\infty;\CC) := \left\{f \text{ measurable on } (0,+\infty) : \underset{t \in (0,+\infty)}{\text{ess sup}} |f(t)| < + \infty\right\},$$
see \cite{da prato-grisvard} p. 325, or \cite{grisvard}, p. 665, Teorema 3, or section 1.14 of \cite{triebel}, where this space is denoted by $(X, D(T_3))_{1{\mbox{-}}\theta,q}$. Note that we can also characterize the space $(D(T_3),X)_{\theta,q}$ taking into account the Osservazione, p. 666, in \cite{grisvard}.

We set also, for any $m \in \NN\setminus\{0\}$
\begin{equation*}
(D(T_3),X)_{m+\theta,q}\;:=\;\left\{\psi\in D(T_3^m) : T_3^m\psi\in (D(T_3),X)_{\theta,q}\right\},
\end{equation*}
and
\begin{equation*}
(X,D(T_3))_{m+\theta,q}\;:=\;\left\{\psi\in D(T_3^m) : T_3^m\psi\in (X,D(T_3))_{\theta,q}\right\},
\end{equation*}
see \cite{lunardi 2}, definition 3.2, p. 64.
\end{Def}

\begin{Rem}\label{Rem Réitération}
Note that for $T_3$ satisfying \eqref{def T3}, $T_3^m$ is closed for any $m \in \NN \setminus \{0\}$ since $\rho(T_3) \neq \emptyset$; consequently, if $m \theta < 1$, we have
$$(D(T_3^m),X)_{\theta,q} = (X,D(T_3^m))_{1-\theta,q} = (X,D(T_3))_{m-m\theta,q} = (D(T_3),X)_{(m-1) + m\theta,q} \subset D(T_3^{m-1}).$$
For more details see \cite{lunardi}, (2.1.13), p. 43; \cite{lunardi 2} Proposition 3.8, p. 69 or \cite{grisvard}, p. 676, Teorema 6.
\end{Rem}

\subsection{Prerequisites}\label{Prerequis}

In this section, we recall some well-known facts, useful in our proofs.

\begin{Lem}[\cite{grisvard}]\label{Lem trace}
Let $T_3$ be a linear operator satisfying \eqref{def T3}. Let $u$ such that
$$u \in W^{n,p}(a,b;X) \cap L^p(a,b;D(T_3^m)),$$ 
where $a,b \in \RR$ with $a<b$, $n,m \in \NN\setminus\{0\}$ and $p \in (1,+\infty)$. Then for any $j \in \NN$ satisfying the Poulsen condition $0<\frac{1}{p}+j < n$ and $s \in \{a,b\}$, we have
$$
u^{(j)}(s) \in (D(T_3^m),X)_{\frac{j}{n}+\frac{1}{np},p}.
$$
\end{Lem}
This result is proved in \cite{grisvard}, Teorema 2', p. 678.

\begin{Lem}\label{Lem triebel}
Let $\psi \in X$ and $T_3$ be a generator of a bounded analytic semigroup in $X$ with $0 \in \rho(T_3)$. Then, for any $m \in \NN\setminus\{0\}$ and $p\in [1,+\infty]$, the next properties are equivalent:
\begin{enumerate}

\item $x\mapsto T_3^m e^{(x-a)T_3} \psi \in L^p(a,+\infty;X)$

\item $\psi \in \left(D(T_3),X\right)_{m-1+\frac{1}{p},p}$

\item $x\mapsto e^{(x-a)T_3} \psi \in W^{m,p}(a,b;X)$

\item $x\mapsto T_3^m e^{(x-a)T_3} \psi \in L^p(a,b;X)$.
\end{enumerate}
\end{Lem}
The equivalence between 1 and 2 is proved in \cite{triebel}. The others are proved in \cite{thorel}, Lemma 3.2, p. 638-639. 

\begin{Lem}[\cite{LMMT}]\label{Lem reg}
Let $V \in \L(X)$ such that $\dis 0 \in \rho(I+V)$. Then 
$$(I+V)^{{\mbox{-}}1} = I - V(I+V)^{-1},$$ 
and $V(I+V)^{-1}(X) \subset V(X)$. Moreover, if $T$ is a linear operator in $X$ such that $V(X) \subset D(T)$ and for $\psi \in D(T)$, \mbox{$TV \psi=VT \psi$}, then 
$$
\forall\,\psi \in D(T), \quad \dis  V(I+V)^{-1}T \psi = T V(I+V)^{-1} \psi.
$$
\end{Lem}
This result is proved in \cite{LMMT}, Lemma 5.1, p. 365.

\section{General fourth order abstract problem}\label{sect linear steady case}

In this section, we study the following more general linear abstract equation of fourth order:
\begin{equation}\label{eq de base}
u^{(4)}(x) + (P+Q) u''(x) + PQu(x) = f(x), \quad x \in (a,b).
\end{equation}
Here, $P$ and $Q$ are two linear operators on a Banach space $X$ verifying some assumptions (see below) and $f\in L^p(a,b;X)$, with $p\in (1,+\infty)$.

We search a classical solution $u$ of \eqref{eq de base}, which is a solution $u$ of \eqref{eq de base} such that 
\begin{equation*}
u \in W^{4,p}(a,b;X) \cap L^p(a,b;D(PQ)) \quad \text{and} \quad u'' \in L^p(a,b;D(P) \cap D(Q)).
\end{equation*}
Moreover, we consider, in this section, the boundary conditions (BCi), $i=1,2,3,4,5$, where $A_0$ is replaced by $P$. We say that $u$ is a classical solution of \eqref{eq de base}-(BCi), $i=1,2,3,4,5$, if $u$ is a classical solution of \eqref{eq de base} satisfying (BCi). 

This study will allow us, in particular, to deduce all the spectral properties for operators $\A_i$, for each $i=1,2,3,4,5$. 

\subsection{Assumptions, main results and some remarks}\label{sect hyp gen}
 
\subsubsection{Assumptions} 

We assume the following hypotheses:

\begin{center}
\begin{tabular}{l}
$(H_1)\quad$ $X$ is a UMD space, \\ \ecart
$(H_2)\quad$ $P$ and $Q$ are closed and $0 \in \rho(P)\cap \rho(Q)$, \\ \ecart
$(H_3)\quad$ $D(P)=D(Q)~$ and $~P^{-1}Q^{-1} = Q^{-1}P^{-1},$ \\ \ecart
$(H_4)\quad$ $-P,-Q \in$ BIP$(X,\theta_0),~$ for $\theta_0 \in [0, \pi)$, \\ \ecart
$(H_5)\quad$ $P-Q$ admits an invertible extension $B \in \L(X)$.
\end{tabular}
\end{center}
Some of our results will need a supplementary hypothesis:
$$(H_6)\quad 0\in \rho(U)\cap \rho(V),$$
where 
\begin{equation} \label{U V gen}
\left\{\begin{array}{lllll}
U &:=& I-e^{c \left( L+M\right) }-B^{-1}\left( L+M\right)^2 \left( e^{c M}-e^{c L}\right) &=& I-T^- \in \L(X) \\ \ecart
V &:=& I-e^{c \left( L+M\right) }+B^{-1}\left( L+M\right)^2 \left( e^{c M}-e^{c L}\right) &=& I-T^+ \in \L(X),
\end{array}\right.
\end{equation}
with $c:=b-a >0$, $L:=-\sqrt{-Q}$ and $M:=-\sqrt{-P}$.

Note that, from $(H_4)$, $-P$ and $-Q$ are sectorial operators; this allows us to define 
\begin{equation}\label{LM}
L=-\sqrt{-Q} \quad \text{and} \quad M=-\sqrt{-P},
\end{equation}
which generate bounded analytic semigroups on $X$ and due to $(H_3)$, $L+M$ also generates a bounded analytic semigroup on $X$, see for instance \cite{pruss-sohr} Theorem 5, p. 443.

\subsubsection{Main results} 

\begin{Th}\label{Th Carre}
Let $f\in L^p(a,b;X)$, with $p \in (1, +\infty)$. Assume that $(H_1)$, $(H_2)$, $(H_3)$, $(H_4)$ and $(H_5)$ hold. Then,
\begin{enumerate}
\item there exists a unique classical solution $u_1$ of problem \eqref{eq de base}-\eqref{condbord1} if and only if 
\begin{equation}\label{reg carre}
\varphi_1, \varphi_2 \in (D(P), X)_{1+\frac{1}{2p},p} \quad \text{and} \quad \varphi_3, \varphi_4 \in (D(P), X)_{\frac{1}{2p},p}.
\end{equation}

This solution, denoted by $F_{\Phi ,f}$ with $\Phi = (\varphi_1, \varphi_2, \varphi_3, \varphi_4)$, is given, for all $x \in [a,b]$, by
\begin{equation}\label{F Th carre}
\begin{array}{rll}
F_{\Phi,f}(x) & = & \dis \left(e^{(x-a)M} - e^{(b-x)M}  e^{cM}\right) Z \varphi_1 + \left(e^{(b-x)M} - e^{(x-a)M} e^{cM}\right) Z \varphi_2 \\ \ecart
&& \dis + \frac{1}{2} \left(e^{(b-x)M}  e^{cM} - e^{(x-a)M}\right) Z M^{-1} \int_a^b e^{(s-a)M} v_0(s)~ ds \\ \ecart
&& \dis + \frac{1}{2} \left(e^{(x-a)M} e^{cM} - e^{(b-x)M} \right) Z M^{-1} \int_a^b e^{(b-s)M} v_0(s)~ ds \\ \ecart
&& \dis + \frac{1}{2} M^{-1} \dis\int_a^x e^{(x-s)M} v_0(s)~ ds + \frac{1}{2} M^{-1} \int_x^b e^{(s-x)M} v_0(s)~ ds,
\end{array}
\end{equation}
where
\begin{equation}\label{v Th carre}
\begin{array}{rll}
v_0(x) & := & \dis \left(e^{(x-a)L} - e^{(b-x)L}  e^{cL}\right) W \left(\varphi_3 + P \varphi_1 \right) \\ \ecart
&&\dis + \left(e^{(b-x)L} - e^{(x-a)L} e^{cL} \right) W \left( \varphi_4 + P \varphi_2 \right) \\ \ecart
&& \dis + \frac{1}{2} \left(e^{(b-x)L} e^{cL} - e^{(x-a)L}\right)W L^{-1} \int_a^b e^{(s-a)L} f(s)~ ds \\ \ecart 
&& \dis + \frac{1}{2} \left( e^{(x-a)L} e^{cL} - e^{(b-x)L} \right) W L^{-1} \int_a^b e^{(b-s)L} f(s)~ ds \\ \ecart
&&\dis + \frac{1}{2} L^{-1} \dis\int_a^x e^{(x-s)L} f(s)~ ds + \frac{1}{2} L^{-1} \int_x^b e^{(s-x)L} f(s)~ ds,
\end{array}
\end{equation}
with $Z := \left(I-e^{2cM} \right)^{-1}$ and $W := \left(I-e^{2cL}\right)^{-1}$.\\
Note that the existence of $Z$ and $W$ is ensured by \cite{lunardi}, Proposition 2.3.6, p. 60.

\item there exists a unique classical solution $u_5$ of problem $\eqref{eq de base}$-$\eqref{condbord5}$ if and only if
\begin{equation}\label{reg phi CB5 carre}
\varphi _{1},\varphi _{2}\in \left( D\left( P\right),X\right) _{1+\frac{1}{2p},p} \quad \text{and} \quad \varphi _{3},\varphi _{4}\in \left( D\left( P\right),X\right) _{\frac{1}{2p},p}.  
\end{equation}
In this case, the unique solution is $u_5 = F_{(\varphi _{1},\varphi_{2},\varphi _{3}-P\varphi _{1},\varphi _{4}-P\varphi _{2}),f}.$
\end{enumerate}
\end{Th}
\begin{Th}\label{thmprinc gen}
Let $f\in L^{p}(a,b;X)$ with $p\in (1,+\infty)$. Assume that $(H_1)$, $(H_2)$, $(H_3)$, $(H_4)$ and $(H_5)$ hold. Then

\begin{enumerate}
\item there exists a unique classical solution $u_2$ of \eqref{eq de base}-\eqref{condbord2} if and only if 
\begin{equation} \label{Reg Data 2}
\varphi_1, \varphi_2 \in \left(D(P),X\right)_{1 + \frac{1}{2} + \frac{1}{2p},p} \quad \text{and} \quad \varphi_3, \varphi_4 \in \left(D(P),X\right)_{\frac{1}{2p},p}. 
\end{equation}
\end{enumerate}
If moreover, $(H_6)$ holds, then
\begin{enumerate}

\item[2.] there exists a unique classical solution $u_3$ of \eqref{eq de base}-\eqref{condbord3} if and only if
\begin{equation}\label{Reg Data 3}
\varphi_1, \varphi_2 \in \left(D(P),X\right)_{1 + \frac{1}{2p},p} \quad \text{and} \quad \varphi_3, \varphi_4 \in \left(D(P),X\right)_{1 + \frac{1}{2} + \frac{1}{2p},p},
\end{equation}

\item[3.] there exists a unique classical solution $u_4$ of \eqref{eq de base}-\eqref{condbord4} if and only if
\begin{equation}\label{Reg Data 4}
\varphi_1, \varphi_2 \in \left(D(P),X\right)_{1 + \frac{1}{2} + \frac{1}{2p},p} \quad \text{and} \quad \varphi_3, \varphi_4 \in \left(D(P),X\right)_{\frac{1}{2p},p}.
\end{equation}
\end{enumerate}
\end{Th}

\subsubsection{Some remarks}

\begin{Rem}\label{Rq1} \hfill
\begin{enumerate}

\item In \cite{LMMT}, problems
\begin{equation*}
\left\{ \begin{array}{l}
u^{(4)}(x) + (2A-kI)u''(x) + (A^2-kA)u(x) = f(x), \quad x \in (a,b), \\ \ecart 
\text{(BCi)},
\end{array}\right.
\end{equation*}
for each $i = 1,2,3,4,5$, correspond to problems \eqref{eq de base}-(BCi) where $P=A$ and $Q=A-kI$, with $k \in \RR$ such that $[k,+\infty) \subset \rho(A)$.

\item Assumptions $(H_2)$ and $(H_3)$ involve: $D(PQ) = D(QP) = D(P^2) = D(Q^2)$. Moreover, 
\begin{equation}\label{commut}
D(LM) = D(ML) = D(L^2) = D(M^2) \quad\text{and}\quad LM = ML.  
\end{equation}

\item Due to $(H_4)$ and \cite{haase}, Proposition 3.2.1, e), p. 71, we have
$$-L,-M \in \text{ BIP}\,(X,\theta_0/2);$$ 
so from \cite{pruss-sohr}, Theorem $4$, p. $441$, $-(L+M) \in$ BIP\,$(X,\theta_0/2+\varepsilon)$ with $\varepsilon >0$. Moreover, $L+M$ is invertible with bounded inverse.

\item Assumption $(H_5)$ means that operators $-P$ and $-Q$ satisfy the following property:
$$\exists B \in \L(X) ~:~ 0\in \rho(B) \quad \text{and} \quad P = Q + B.$$
When $P=A$ and $Q=A-\mu I$ with $\mu \in \CC\setminus\{0\}$, then $P-Q=\mu I \in \L(X)$ is invertible. 

\item Assume that $(H_1)$, $(H_2)$, $(H_3)$, $(H_4)$ and $(H_5)$ are satisfied.

Then, $(H_6)$ holds if $c=b-a$ is enough large. Indeed, since $L$, $M$ and $L+M$ are invertible with bounded inverse and generate bounded analytic semigroups, there exist $\delta >0$ and $C \geqslant 1$ (see \cite{lunardi}, (2.1.1) and (2.1.2) p. $35$ taking $\omega =-\delta $ where $\delta >0$ is enough small) such that
\begin{equation*}
\max \left( \left\Vert e^{c\left( L+M\right)} \right\Vert _{\mathcal{L}
\left( X\right) },\left\Vert M^{2}e^{cM}\right\Vert _{\mathcal{L}\left(
X\right) },\left\Vert L^{2}e^{cL}\right\Vert _{\mathcal{L}\left( X\right)
}\right) \leqslant Ce^{-\delta c},
\end{equation*}
thus
\begin{eqnarray*}
M_B & := & \max \left( \left\Vert T^-\right\Vert _{\mathcal{L}\left( X\right)},\left\Vert T^+\right\Vert_{\mathcal{L}(X) }\right) \\ \\
&\leqslant &~~\,\left\Vert e^{c \left( L+M\right) }\right\Vert _{\L(X)} + \left\Vert B^{-1}\left( L+M\right)^{2}M^{-2}\right\Vert _{\mathcal{L}\left( X\right) }\left\Vert
M^{2}e^{c M}\right\Vert _{\mathcal{L}\left( X\right) } \\ \ecart
&&+\left\Vert B^{-1}\left( L+M\right) ^{2}L^{-2}\right\Vert _{\mathcal{L}
\left( X\right) }\left\Vert L^{2}e^{c L}\right\Vert _{\mathcal{L}\left( X\right) } \\ \\
&\leqslant &\left(1 + \left\Vert B^{-1}\left( L+M\right)^{2}M^{-2} \right\Vert _{\mathcal{L}\left( X\right) }+\left\Vert B^{-1}\left(
L+M\right) ^{2}L^{-2}\right\Vert _{\mathcal{L}\left( X\right) }\right)
C e^{-\delta c }.
\end{eqnarray*}
Finally, for $c=b-a$ enough large, $\left\Vert T^-\right\Vert _{\L(X)}<1$ and $\left\Vert T^+\right\Vert _{\L(X)}<1$. It follows that $U$ and $V$ are invertible.

\item In some particular cases, we could check assumption $(H_6)$ using functional calculus for sectorial operators, see for instance \cite{LMMT} where $B = k \in \RR\setminus \{0\}$.
\end{enumerate}
\end{Rem}

\begin{Rem}\label{P-Q inclus dans T}\hfill

\begin{enumerate}

\item Note that $P-Q\subset B \in \L(X)$ involves that $P-Q$ is closable but, in general, is not closed. In fact, if $P-Q$ is closed, then due to $\overline{D(P-Q)} = X$ and $P-Q \subset B \in \L(X)$, we deduce that $D(P-Q) = X$, thus $D(P) = D(Q) = X$. Then, $P-Q$ is closed only if $P-Q \in \L(X)$. 

\item Since $L=-\sqrt{-Q}$ and $M=-\sqrt{-P}$ we have $P-Q=L^{2}-M^{2}$. Moreover%
\begin{equation}
(L-M)(L+M)=L^{2}-M^{2},  \label{L-M  L+M}
\end{equation}
indeed, it is well known that if $C_{j}$, $j=1,2,3,4,$ are linear operators,
then
$$C_{1}C_{3}+C_{1}C_{4}+C_{2}C_{3}+C_{2}C_{4}\subset \left( C_{1}+C_{2}\right)
\left( C_{3}+C_{4}\right),$$
but not necessary the equality. In our case, from \eqref{commut}, we have
$$L^{2}-M^{2}=L^{2}+LM-ML-M^{2}\subset \left( L-M\right) \left( L+M\right),$$
and also $D(L^{2}-M^{2})=D\left( M^{2}\right) $. To conclude it suffices to
check that 
$$D((L-M)(L+M))\subset D(L^{2}-M^{2})=D\left( M^{2}\right).$$
To this end, consider $\psi\in D((L-M)(L+M))$. Then $\psi\in D(L+M)$ and 
$$\left(L+M\right) \psi\in D\left( L-M\right) =D\left( M\right).$$ 
Thus, there exists $\chi\in X$ such that $\left( L+M\right) \psi=M^{-1}\chi$. Hence $M\psi=(L+M)^{-1}\chi\in D(M)$, that is $\psi\in D(M^{2}) = D(L^2)$.

\item Recall that $L^2-M^2 \subset B$ means that
$$\forall~ \psi \in D(M^2), \quad(L^2-M^2) \psi = B \psi.$$ 
Hence
\begin{equation} \label{(L-M)y=}
\forall~ \psi\in D(M), \quad (L-M) \psi = (P-Q)(L+M)^{-1} \psi = B (L+M)^{-1} \psi. 
\end{equation}

\item For $z\in \rho(-P) \cap \rho(-Q)$, we have
\begin{eqnarray*}
B (-Q-zI)^{-1} (-P-zI)^{-1} & = & (P-Q) (-Q-zI)^{-1} (-P-zI)^{-1} \\
& = & (-Q-zI+P+zI) (-Q-zI)^{-1} (-P-zI)^{-1} \\
& = & (-P-zI)^{-1} - (-Q-zI)^{-1}.
\end{eqnarray*}
Thus
\begin{equation}\label{resolv}
B (-Q-zI)^{-1} (-P-zI)^{-1} = (-P-zI)^{-1} - (-Q-zI)^{-1}.  
\end{equation}
\end{enumerate}
\end{Rem}

\subsection{Preliminary results} \label{sect prel}

\subsubsection{Particular solutions : proof of \refT{Th Carre}}

\begin{enumerate}
\item The proof is similar to the one of Theorem 2.2, p. 355 in \cite{LMMT}, thus we omit it.

\item It suffices to remark that condition \eqref{condbord5} writes as
\begin{equation*}
\left\{\begin{array}{rclrcl}
u(a) &= & \varphi_1, & u(b) &= &\varphi_2, \\
u''(a) &= &\varphi_3 - P \varphi_1, & u''(b) &= &\varphi_4 - P \varphi_2,
\end{array}\right.
\end{equation*}
and then to apply the first statement.
\end{enumerate}

\subsubsection{Representation formula} \label{Repre gen}

We begin by two technical lemmas. 

\begin{Lem} Assume that $(H_1)$, $(H_2)$, $(H_3)$, $(H_4)$ and $(H_5)$ hold. Then, $P^{-1}B = BP^{-1}$, which means that
\begin{equation}\label{commut P et T}
\forall ~\psi \in D(P),\quad B \psi\in D(P) \quad \text{and} \quad PB\psi = BP\psi.
\end{equation}
In the same way, we have $Q^{-1}B=BQ^{-1}$, which means that
\begin{equation}\label{commut Q et T}
\forall~ \psi\in D(Q), \quad B\psi\in D(Q) \quad \text{and} \quad QB\psi = BQ\psi.  
\end{equation}
\end{Lem}

\begin{proof}
First of all, note that from $(H_3)$, $P-Q$ is resolvent commuting with $P$, $Q$, $P+Q$ and all linear combination of $P$ and $Q$.

Let $\chi\in X$. There exists $\left( \chi_{n}\right) _{n\geqslant 0}\subset
D\left( P\right)$ such that $\chi_{n}\underset{n\longrightarrow +\infty }{\longrightarrow }\chi$. Then, since we have $P^{-1}B,BP^{-1}\in \mathcal{L}\left( X\right)$ and $B=P-Q$ on $D\left( P\right)$, we obtain
$$\begin{array}{lllllll}
P^{-1}B \chi &=&\underset{n\longrightarrow +\infty }{\lim }P^{-1}B \chi_{n} &=&\underset{n\longrightarrow +\infty }{\lim }P^{-1}\left( P-Q\right) \chi_{n} &=&\underset{n\longrightarrow +\infty }{\lim }\left( P-Q\right) P^{-1} \chi_{n} \\ \ecart
&=&\underset{n\longrightarrow +\infty }{\lim }BP^{-1}\chi_{n} &=&BP^{-1}\chi.
\end{array}$$
Thus, $P^{-1}B=BP^{-1}$.

If $\psi\in D(P)$ then, setting $\chi=P\psi$, we have 
\begin{equation*}
B\psi=BP^{-1}\chi=P^{-1}B\chi=P^{-1}BP\psi.
\end{equation*}
Then $B\psi=P^{-1}BP\psi \in D(P)$ and $PB\psi=BP\psi$. In the same way, replacing $P$ by $Q$ in the previous proof, we obtain \eqref{commut Q et T}.
\end{proof}

\begin{Lem} Assume that $(H_1)$, $(H_2)$, $(H_3)$, $(H_4)$ and $(H_5)$ hold. Then, $D(L) = D(M)$.
\end{Lem}

\begin{proof}
By definition, using the Dunford-Riesz integral, we have
\begin{equation*}
\left\{\begin{array}{lll}
\dis -M = \sqrt{-P} & = & \dis \frac{1}{2i\pi} \left( I-P\right) \int_{\gamma} \frac{\sqrt{z}}{1+z} \left(-P-zI\right)^{-1}~dz \\ \ecart
\dis -L = \sqrt{-Q} & = & \dis\frac{1}{2i\pi} \left( I-Q\right) \int_{\gamma} \frac{\sqrt{z}}{1+z} \left(-Q-zI\right)^{-1}~dz,
\end{array}\right.
\end{equation*}
where $\gamma$ is a sectorial curve surrounding $\sigma \left(-P\right) \cup \sigma \left(-Q\right)$, see \cite{haase}, p. 61. Moreover,
\begin{equation*}
\left\{\begin{array}{lll}
\dis \int_{\gamma }\frac{\sqrt{z}}{1+z}\left(-P-zI\right)^{-1}~dz \in \L(X) \\ \ecart
\dis \int_{\gamma }\frac{\sqrt{z}}{1+z}\left(-Q-zI\right)^{-1}~dz \in \L(X),
\end{array}\right.
\end{equation*} 
and 
\begin{equation}\label{syst int P/Q in L(X)}
\left\{\begin{array}{lll}
\dis \psi \in D(M) = D(\sqrt{-P}) &\Longleftrightarrow &\dis\int_{\gamma }\frac{\sqrt{z}}{1+z}\left( -P-zI\right) ^{-1}\psi~ dz\in
D\left( P\right) \\ \ecart

\dis \psi \in D(L) = D(\sqrt{-Q}) &\Longleftrightarrow &\dis\int_{\gamma }\frac{\sqrt{z}}{1+z}\left( -Q-zI\right) ^{-1}\psi~ dz\in
D\left( P\right).
\end{array}\right.
\end{equation}
Now, we will show that $D(L) \subset D(M)$. Let $\psi \in D(L)$, then we have 
\begin{eqnarray*}
\int_{\gamma }\frac{\sqrt{z}}{1+z}\left( -Q-zI\right) ^{-1}\psi ~dz & = & \int_{\gamma }\frac{\sqrt{z}}{1+z}\left( -P-zI\right) \left(-Q-zI\right) ^{-1}\left( -P-zI\right) ^{-1}\psi ~dz \\
&=&\int_{\gamma }\frac{\sqrt{z}}{1+z}\left( -P+Q-Q-zI\right) \left(
-Q-zI\right) ^{-1}\left( -P-zI\right) ^{-1}\psi ~dz \\
&=&\int_{\gamma }\frac{\sqrt{z}}{1+z}\left( -P+Q\right) \left( -Q-zI\right)
^{-1}\left( -P-zI\right) ^{-1}\psi ~dz \\
&&+\int_{\gamma }\frac{\sqrt{z}}{1+z}\left( -P-zI\right) ^{-1}\psi ~dz \\
&=&-\int_{\gamma }\frac{\sqrt{z}}{1+z}B\left( -Q-zI\right) ^{-1}\left(-P-zI\right) ^{-1}\psi~ dz \\
&&+\int_{\gamma }\frac{\sqrt{z}}{1+z}\left(-P-zI\right) ^{-1}\psi~ dz \\
&=&-BQ^{-1}\zeta +\int_{\gamma }\frac{\sqrt{z}}{1+z}\left( -P-zI\right)^{-1}\psi ~dz,
\end{eqnarray*}
where
$$\zeta := \int_\gamma \frac{\sqrt{z}}{1+z}Q (-Q-zI)^{-1} (-P-zI)^{-1}\psi ~dz\in X.$$ 
Then, since $\psi \in D(L)$ and due to \eqref{syst int P/Q in L(X)}, we have
\begin{equation*}
\int_\gamma \frac{\sqrt{z}}{1+z} (-P-zI)^{-1} \psi~ dz =  Q^{-1}B\zeta + \int_\gamma \frac{\sqrt{z}}{1+z} (-Q-zI)^{-1}\psi~ dz \in D(Q)=D(P),
\end{equation*}
Thus, from \eqref{syst int P/Q in L(X)}, we deduce that $\psi \in D(\sqrt{-P}) = D(M)$. 

Replacing $P$ by $Q$ in the previous proof, we show that $D(M)\subset D(L)$.
\end{proof}

Now, we give an explicit representation formula of the classical solution $u$ of equation \eqref{eq de base}.

\begin{Prop}\label{Proposition 1}
Let $f\in L^p(a,b;X)$, $p \in (1, +\infty)$. Assume that $(H_1)$, $(H_2)$, $(H_3)$, $(H_4)$ and $(H_5)$ hold. If $u$ is a classical solution of \eqref{eq de base}, then there exists $K_1, K_2, K_3, K_4 \in X$, such that for all $x\in [a,b]$, we have
\begin{equation}\label{Repre u}
u(x) = e^{(x-a)M} K_1 + e^{(b-x)M} K_2 + e^{(x-a)L} K_3 + e^{(b-x)L} K_4 + F_{0 ,f}(x),
\end{equation}
where $F_{0,f}$ is defined in \refT{Th Carre} with $\Phi = 0 = (0,0,0,0)$.
\end{Prop}

\begin{proof}
If $u$ is a classical solution of \eqref{eq de base}, due to \refT{Th Carre}, we can take the classical solution $F_{0,f}$ of \eqref{eq de base}-\eqref{condbord1} as a particular solution; $i.e.$
\begin{equation}\label{condbord1homogene gen}
F_{0,f}(a)=F_{0,f}(b)=F_{0,f}''(a)=F_{0,f}''(b) = 0.
\end{equation}
Then $u_{hom}:=u-F_{0 ,f}$ is a classical solution of 
\begin{equation*}
u^{(4)}(x)+(P+Q)u''(x)+PQu(x)=0, \quad x \in (a,b).
\end{equation*}
We set
\begin{equation*}
v := -QB^{-1}u_{hom}-B^{-1}u_{hom}'', \quad \text{and} \quad w := PB^{-1}u_{hom}+B^{-1}u_{hom}''.
\end{equation*}
Moreover, from \eqref{commut Q et T}, we have
$$
v''+Pv = - B^{-1}\left( u_{hom}^{(4)}+(P+Q)u_{hom}'' + PQ u_{hom} \right)  =  0,
$$
and
$$
w''+Qw = B^{-1}\left( u_{hom}^{(4)} + (P+Q) u_{hom}'' + PQ u_{hom}\right) = 0.
$$
Note that $u_{hom}'' \in L^p(a,b;D(P))$ then, from \eqref{commut P et T}, $PB u_{hom}'' = BP u_{hom}''$ in $L^p(a,b;X)$. In the same way $QB u_{hom}'' = BQ u_{hom}''$. Moreover, since $u_{hom}\in L^{p}(a,b;D(PQ))$, we have
\begin{equation*}
QPB u_{hom} = BQP u_{hom} = PQB u_{hom} = BPQ u_{hom}.
\end{equation*}
Then, from \cite{falamaintaya}, there exist $K_1, K_2, K_3, K_4 \in X$ such that
\begin{equation*}
v(x) = e^{(x-a)M} K_1 + e^{(b-x)M} K_2 \quad \text{and} \quad w(x) = e^{(x-a)L} K_3 + e^{(b-x)L} K_4.
\end{equation*}
Finally, since $v+w = (P-Q) B^{-1}u_{hom} = u_{hom}$, we obtain \eqref{Repre u}.
\end{proof}

\subsubsection{Regularity of the difference of analytic semigroups}\label{section reg diff semi-groupe}

\begin{Def}\label{Def MR}
A linear operator $\Lambda$ on $X$, satisfies maximal regularity property $(\mathcal{MR})$ if and only if: there exists $q \in (1,+\infty)$ and $a,b \in \RR$ with $a<b$ such that, for all $h\in L^{q}(a,b;X)$, there exists a unique $u\in W^{1,q}(a,b;X) \cap L^q(a,b;D(\Lambda))$ satisfying
\begin{equation*}
\left\{ 
\begin{array}{l}
u'(x)=\Lambda u(x)+h(x),~~\text{a.e.}~x\in (a,b) \\ 
u(a)=0.
\end{array}
\right.
\end{equation*}
\end{Def}
Due to \cite{dore}, Theorem 2.2, Theorem 2.4, Theorem 4.2, and \cite{dore-venni}, Theorem 3.2, we have 
\begin{enumerate}
\item $\Lambda$ satisfies $(\mathcal{MR})$ implies that $\Lambda$ is the infinitesimal generator of an analytic semigroup.

\item $(\mathcal{MR})$ is independent of $q$, $a$ and $b$.

\item $-\Lambda \in $ BIP\,$(X,\theta),~0<\theta <\pi /2$ involves that $\Lambda $ satisfies $(\mathcal{MR})$.
\end{enumerate}

Consider $\Lambda_1$, $\Lambda_2$ and $\B$ three linear operators on $X$ such that
\begin{equation*}
\left\{ \begin{array}{l}
\Lambda_1\text{ satisfies }\mathcal{(MR)} \\ 
0\in \rho \left( \Lambda_1\right)  \\ 
\B\in \mathcal{L}\left( X\right) \\
\Lambda_1 \B = \B \Lambda_1\text{ on }D(\Lambda_1)\text{ (commutative case)} \\ 
\Lambda_2 = \Lambda_1 + \B.
\end{array}\right. 
\end{equation*}
Then, from \cite{tanabe}, Theorem 3.4.1, p. 71 or \cite{engel-nagel}, Chapter III section 1.3, p. 158, we deduce that $\Lambda_2$ is the infinitesimal generator of an analytic semigroup. 

In the sequel, for a linear operator $T$ on $X$, we set
$$D(T^\infty) = \bigcap_{n\in \NN} D(T^n).$$

\begin{Th}\label{Reg Diff exp}
Let $\psi \in X$. For $x \in (a,b) \subset \RR$, with $a<b$, we set
\begin{equation*}
u_\psi (x) = \left( e^{(x-a) \Lambda_2} - e^{(x-a) \Lambda_1}\right) \psi \quad \text{and} \quad v_\psi (x) = \left( e^{(b-x) \Lambda_2} - e^{(b-x) \Lambda_1}\right) \psi.
\end{equation*}

Then, for all $m,q\in \NN \setminus \left\{0,1\right\}$, the following properties hold:

\begin{enumerate}
\item $u_{\psi }\in C^{\infty }\left((a,b];X\right) \cap C^{0}\left(
[a,b];X\right)$, moreover for $\ell \geqslant 1$ and $x\in (a,b]$
\begin{equation*}
u_{\psi }^{\left( \ell \right) }\left( x\right) =\Lambda_1^{\ell }u_{\psi }\left(
x\right) + T_{\ell }\Lambda_1^{\ell -1}e^{\left( x-a\right) \Lambda_2}\B\psi,
\end{equation*}
where $T_{\ell }=\sum\limits_{k=1}^{\ell }\Lambda_2^{k-1}\Lambda_1^{-\left( k-1\right)}\in \L(X).$

Note that $\Lambda_2^{0}=\Lambda_1^{0}=I$ then in particular $T_{1}=I$.

\item $u_{\psi }\in W^{1,p}\left( a,b;X\right) \cap L^{p}\left(
a,b;D(\Lambda_1)\right)$.

\item $u_{\psi }\in W^{m,p}\left( a,b;X\right) \cap L^{p}\left(
a,b;D(\Lambda_1^{m})\right)$ if and only if  
\begin{equation*}
\B\psi \in \left(D\left(\Lambda_1^{m-1}\right),X\right)_{\frac{1}{\left(m-1\right) p},p},
\end{equation*}

\item If $u_{\psi }\in W^{m,p}\left( a,b;X\right) \cap L^{p}\left(
a,b;D(\Lambda_1^{m})\right)$, then for all $\ell \in \left\{ 1,...,m-1\right\}$,
\begin{equation*}
u_{\psi }^{\left( \ell \right) }\in W^{m-\ell ,p}\left( a,b;X\right) \cap
L^{p}\left( a,b;D(\Lambda_1^{m-\ell })\right) .
\end{equation*}

\item If $\B$ satisfies $\B\left( X\right) \subset \left( D\left(\Lambda_1^{q-1}\right) ,X\right) _{\frac{1}{\left( q-1\right) p},p}$, with $q\in(1,+\infty)$, then for all $\psi \in X$:
\begin{equation*}
u_{\psi }\in W^{q,p}\left( a,b;X\right) \cap L^{p}\left( a,b;D(\Lambda_1^{q})\right).
\end{equation*}

\item If moreover $\B$ satisfies
\begin{enumerate}
\item $\B(X) \subset D(\Lambda_1^{q})\text{ then }\Lambda_1^{q} \B \in \mathcal{L}\left(X\right)$

\item $0\in \rho \left( \Lambda_1^{q} \B \right),$
\end{enumerate}
then, we have:
\begin{enumerate}
\item $\forall~\psi \in X : u_{\psi }\in W^{q+1,p}\left( a,b;X\right) \cap
L^{p}\left( a,b;D(\Lambda_1^{q+1})\right) $

\item Let $m\geqslant q+2$, with $q\in(1,+\infty)$. Then $u_{\psi }\in W^{m,p}\left( a,b;X\right) \cap
L^{p}\left( a,b;D(\Lambda_1^{m})\right)$ if and only if
\begin{equation*}
\psi \in \left( D\left( \Lambda_1^{m-q-1}\right) ,X\right) _{\frac{1}{\left(
m-q-1\right) p},p}.
\end{equation*}
\end{enumerate}
\item The six previous statements hold if we replace $u_\psi$ by $v_\psi$.
\end{enumerate}
\end{Th}

\begin{proof}\hfill

\begin{enumerate}
\item Note that for $x>a$, we have $e^{\left( x-a\right) \Lambda_1}\psi,~ e^{\left( x-a\right) \Lambda_2}\psi \in D\left( \Lambda_1^{\infty }\right)=D\left(\Lambda_2^\infty\right)$.

Moreover, $u_{\psi }\in C^{\infty }\left( (a,b];X\right) \cap C^{0}\left(
[a,b];X\right) $. Then, for $\ell \geqslant 1$ and $x\in (a,b]$, it follows
\begin{eqnarray*}
u_{\psi }^{\left( \ell \right) }\left( x\right) &=&\Lambda_2^{\ell }e^{\left(
x-a\right) \Lambda_2}-\Lambda_1^{\ell }e^{\left( x-a\right) \Lambda_1}\psi \\
&=&\Lambda_1^{\ell }\left( e^{\left( x-a\right) \Lambda_2}-e^{\left( x-a\right) \Lambda_1}\right)
\psi +\left( \Lambda_2^{\ell }-\Lambda_1^{\ell }\right) e^{\left( x-a\right) \Lambda_2}\psi \\
&=&\Lambda_1^{\ell }u_{\psi }\left( x\right) +\left( \sum\limits_{k=1}^{\ell
}\Lambda_2^{k-1}\Lambda_1^{\ell -k}\right) \left( \Lambda_2-\Lambda_1\right) e^{\left( x-a\right) \Lambda_2}\psi \\
&=&\Lambda_1^{\ell }u_{\psi }\left( x\right) +\left( \sum\limits_{k=1}^{\ell
}\Lambda_2^{k-1}\Lambda_1^{\ell -k}\right) e^{\left( x-a\right) \Lambda_2}\B\psi \\
&=&\Lambda_1^{\ell }u_{\psi }\left( x\right) +\left( \sum\limits_{k=1}^{\ell
}\Lambda_2^{k-1}\Lambda_1^{-\left( k-1\right) }\right) \Lambda_1^{\ell -1}e^{(x-a)\Lambda_2} \B \psi.
\end{eqnarray*}

\item From the previous statement, $u_{\psi }$ is a solution $C^{1}\left( (a,b];X\right) \cap C^{0}\left( [a,b];X\right) $ of the Cauchy problem
\begin{equation*}
\left\{ 
\begin{array}{llll}
u^{\prime }\left( x\right) &=& \Lambda_1 u\left( x\right) +e^{\left( x-a\right) \Lambda_2} \B \psi, & x\in (a,b] \\ 
u(a) &=& 0.
\end{array}%
\right.
\end{equation*}

Thus, for $x\in (a,b]$, $u$ is given by the variation of constant formula
\begin{equation}\label{u phi = gen}
u_{\psi }\left( x\right) = \int_{a}^{x}e^{\left( x-s\right) \Lambda_1}e^{\left( s-a\right) \Lambda_2} \B \psi ~ds = \Lambda_1^{-1}\left[\Lambda_1 \int_{a}^{x}e^{\left( x-s\right) \Lambda_1}e^{\left(s-a\right) \Lambda_2} \B \psi ~ds\right].
\end{equation}
Moreover, we have 
\begin{equation*}
s\mapsto e^{\left( s-a\right) \Lambda_2} \B \psi \in C^{0}\left( [a,b];X\right) \subset L^{p}\left( a,b;X\right),
\end{equation*}
and from \cite{dore-venni}, Theorem 3.2, p. 196, it follows
\begin{equation*}
g_{1,\psi }:x\mapsto \Lambda_1 \int_{a}^{x}e^{\left( x-s\right) \Lambda_1}e^{\left(s-a\right) \Lambda_2} \B \psi ~ds\in L^{p}\left( a,b;X\right).
\end{equation*}
We deduce that $u_{\psi }=\Lambda_1^{-1}g_{1,\psi }\in L^{p}\left(a,b;D(\Lambda_1)\right)$ and then
\begin{equation*}
u^{\prime }=\Lambda_1 u+e^{\left( \cdot -a\right) \Lambda_2} \B \psi \in L^{p}\left( a,b;X\right).
\end{equation*}

\item Assume that $\B \psi \in \left( D\left( \Lambda_1^{m-1}\right) ,X\right) _{\frac{1}{\left( m-1\right) p},p}$. From \eqref{u phi = gen}, we have
\begin{equation*}
u_{\psi }\left( x\right) = \Lambda_1^{-m}\left[ \Lambda_1 \int_{a}^{x}e^{\left(x-s\right) \Lambda_1}\Lambda_1^{m-1}e^{\left( s-a\right) \Lambda_2} \B \psi ~ds\right] , \quad x\in (a,b].
\end{equation*}

Since $\B \psi \in \left( D\left( \Lambda_1^{m-1}\right) ,X\right) _{\frac{1}{\left( m-1\right) p},p}$, we obtain
\begin{equation*}
s\mapsto \Lambda_1^{m-1}e^{\left( s-a\right) \Lambda_1} \B \psi \in L^{p}\left( a,b;X\right).
\end{equation*}
From \cite{triebel}, Theorem, p. 96, we deduce that
\begin{equation*}
s\mapsto \Lambda_1^{m-1}e^{\left( s-a\right) \Lambda_2} \B \psi =e^{\left( s-a\right)\B}\Lambda_1^{m-1}e^{\left( s-a\right) \Lambda_1} \B \psi \in L^{p}\left( a,b;X\right),
\end{equation*}
and from \cite{dore-venni}, Theorem 3.2, p. 196, we have
\begin{equation*}
g_{m,\psi }:x\mapsto \Lambda_1\int_{a}^{x}e^{\left( x-s\right)\Lambda_1}\Lambda_1^{m-1}e^{\left( s-a\right) \Lambda_2} \B \psi~ds \in L^p(a,b;X).
\end{equation*}
It follows that
$$u_{\psi }=\Lambda_1^{-m}g_{m,\psi }\in L^{p}\left(a,b;D(\Lambda_1^{m})\right).$$

Moreover, $u_{\psi }\in W^{m,p}\left( a,b;X\right) $ since $u_{\psi }\in C^{1}\left( (a,b];X\right) \cap C^{0}\left( [a,b];X\right) $ and from statement~1., for $\ell \in \left\{ 1,...,m\right\} $ and $x\in (a,b]$, we obtain
\begin{equation*}
u_{\psi }^{\left( \ell \right) }\left( x\right) = \Lambda_1^{\ell }u_{\psi }\left(x\right) + T_{\ell }\Lambda_1^{\ell -1}e^{\left( x-a\right) \Lambda_2} \B \psi.
\end{equation*}
Then $u_{\psi }^{\left( \ell \right) }\in L^{p}\left( a,b;X\right) $.

Conversely, if $u_{\psi }\in W^{m,p}\left( a,b;X\right)
\cap L^{p}\left( a,b;D(\Lambda_1^{m})\right) $ then, from \refL{Lem trace}, we have
\begin{eqnarray*}
\B \psi =u'_\psi (a) \in \left( D\left( \Lambda_1^{m}\right) ,X\right) _{\frac{1}{mp}+\frac{1}{m},p}
&=&\left( X,D\left( \Lambda_1^{m}\right) \right) _{1-\frac{1}{mp}-\frac{1}{m},p} \\
&=&\left( X,D\left( \Lambda_1\right) \right) _{m-\frac{1}{p}-1,p} \\
&=&\left( D\left( \Lambda_1^{m-1}\right) ,X\right) _{\frac{1}{\left( m-1\right) p},p}.
\end{eqnarray*}

\item Note first that, from statement 3., we have $\B \psi \in \left( D\left(\Lambda_1^{m-1}\right) ,X\right) _{\frac{1}{\left( m-1\right) p},p}$, then 
\begin{equation*}
\Lambda_1^{m-1}e^{\left( \cdot -a\right) \Lambda_1}\Lambda_2 \psi \in L^{p}\left( a,b;X\right) .
\end{equation*}
Let $\ell \in \left\{ 1,...,m-1\right\} $. For $x\in (a,b]$, we have
\begin{equation*}
u_{\psi }^{\left( \ell \right) }\left( x\right) = \Lambda_1^{\ell }u_{\psi }\left(
x\right) + T_{\ell }\Lambda_1^{\ell -1}e^{\left( x-a\right) \Lambda_2}\B\psi.
\end{equation*}
It follows that:
\begin{enumerate}
\item $u_{\psi }^{\left( \ell \right) }\in L^{p}(a,b;D(\Lambda_1^{m-\ell }))$ since
\begin{equation*}
\Lambda_1^{m-\ell }u_{\psi }^{\left( \ell \right) }=\Lambda_1^{m}u_{\psi }\left( x\right) + T_{\ell }\Lambda_1^{m-1} e^{\left( x-a\right) \Lambda_2} \B \psi \in L^{p}\left( a,b;X\right).
\end{equation*}

\item $u_{\psi }^{\left( \ell \right) }\in W^{m-\ell ,p}\left( a,b;X\right)$ since, for $k\in \left\{ 1,...,m-\ell \right\}$, we have $\left( u_{\psi }^{\left( \ell \right) }\right) ^{\left( k\right) }\in
L^{p}\left( a,b;X\right)$ and 
\begin{eqnarray*}
\left( u_{\psi }^{\left( \ell \right) }\right) ^{\left( k\right) } &=& u_{\psi }^{\left( \ell +k\right) } \\ 
&=&\Lambda_1^{\ell +k}u_{\psi }+T_{\ell +k}\Lambda_1^{\ell +k-1}e^{\left( \cdot -a\right)\Lambda_2}\B \psi , \\
&=&\Lambda_1^{\ell +k}u_{\psi }+T_{\ell +k}\Lambda_1^{\ell +k-m}\Lambda_1^{m-1}e^{\left( \cdot -a\right) \Lambda_2} \B \psi \in L^{p}\left( a,b;X\right).
\end{eqnarray*}
Since $\Lambda_1^{\ell +k}u_{\psi }\in L^{p}\left( a,b;X\right)$ and $\ell
+k\leqslant m$, we obtain $T_{\ell +k}\Lambda_1^{\ell +k-m}\in \mathcal{L}\left( X\right)$ and
\begin{equation*}
T_{\ell +k}\Lambda_1^{\ell +k-m}\Lambda_1^{m-1}e^{\left( \cdot -a\right) \Lambda_2}\B\psi \in L^{p}\left( a,b;X\right).
\end{equation*}
\end{enumerate}

\item If $\psi \in X$, then $\B\psi \in \left( D\left(\Lambda_1^{q-1}\right) ,X\right) _{\frac{1}{\left( q-1\right) p},p}$ and from statement 3., we obtain the result.

\item For statement a) it suffice to note that if $\psi \in X$, then
\begin{equation*}
\B\psi \in D\left( \Lambda_1^{q}\right) \subset D\left( \Lambda_1^{q-1}\right) \quad \text{and} \quad \Lambda_1^{q-1}\B\psi = D\left( \Lambda_1\right) \subset \left( D\left( \Lambda_1\right) ,X\right) _{\frac{1}{p},p}.
\end{equation*}
From statement 3., $u_{\psi }\in W^{q+1,p}\left( a,b;X\right) \cap L^{p}\left(a,b;D(\Lambda_1^{q+1})\right)$.

For statement b), if $\psi \in \left( D\left( \Lambda_1^{m-q-1}\right),X\right) _{\frac{1}{\left( m-q-1\right) p},p}$ then, from the reiteration property described in \refR{Rem Réitération}, we have
\begin{equation*}
\psi \in D\left( \Lambda_1^{m-q-2}\right) \quad \text{and} \quad \Lambda_1^{m-q-2}\psi \in \left(D\left( \Lambda_1\right) ,X\right) _{\frac{1}{p},p}.
\end{equation*}
It follows that
$$\B\psi = \Lambda_1^{q}\B\Lambda_1^{-q}\psi \in D\left( \Lambda_1^{m-2}\right),$$
and
$$\Lambda_1^{m-2}\B\psi = \Lambda_1^{m-2}\Lambda_1^{q}\B\Lambda_1^{-q}\psi = \Lambda_1^{q}\B\Lambda_1^{m-q-2}\psi \in \left(D\left( \Lambda_1\right) ,X\right) _{\frac{1}{p},p}.$$
Then, from statement 3., we obtain that $u_{\psi }\in W^{m,p}\left( a,b;X\right) \cap L^{p}\left(a,b;D(\Lambda_1^{m})\right)$.

Conversely, if $u_{\psi }\in W^{m,p}\left( a,b;X\right)\cap L^{p}\left( a,b;D(\Lambda_1^{m})\right)$, then, from statement 3., we have $\B \psi \in \left( D\left( \Lambda_1^{m-1}\right) ,X\right) _{\frac{1}{\left( m-1\right) p},p}$. Hence, from \refR{Rem Réitération}, we deduce that
\begin{equation*}
\B \psi \in D\left(\Lambda_1^{m-2}\right) \quad \text{and} \quad \Lambda_1^{m-2}\B\psi \in \left( D\left(\Lambda_1\right) ,X\right) _{\frac{1}{p},p}.
\end{equation*}
Then, there exists $\chi \in \left( D\left( \Lambda_1\right) ,X\right) _{\frac{1}{p},p}$ such that $\B\psi = \Lambda_1^{-m+2}\chi $, it follows
$$\psi =\left( \Lambda_1^{q}\B\right) ^{-1}\Lambda_1^{q}\B\psi =\left( \Lambda_1^{q}\B\right)^{-1}\Lambda_1^{-m+q+2}\chi \in D\left( \Lambda_1^{m-q-2}\right),$$
and
$$\Lambda_1^{m-q-2}\psi =\left( \Lambda_1^{q}\B\right) ^{-1}\chi \in \left( D\left(\Lambda_1\right),X\right) _{\frac{1}{p},p},$$
which gives $\psi \in \left( D\left( \Lambda_1^{m-q-1}\right) ,X\right)_{\frac{1}{\left( m-q-1\right) p},p}$.

\item It suffice to write $v_\psi (x) = u_\psi \left(b+a-x\right)$. Then, $v_\psi$ satisfies the properties of $u_\psi$.
\end{enumerate}
\end{proof}

\begin{Cor}\label{Cor diff exp}
Assume that $(H_1)$, $(H_2)$, $(H_3)$, $(H_4)$ and $(H_5)$ are satisfied. Let $\psi \in X$. For $x \in (a,b) \subset \RR$, we set
\begin{equation*}
u_{\psi }\left( x\right) =\left( e^{(x-a) L}-e^{(x-a) M}\right) \psi \quad \text{and} \quad v_\psi (x) = \left( e^{(b-x) L} - e^{(b-x) M}\right) \psi.
\end{equation*}
Let $m\geqslant 3$, then 

\noindent1. $\forall~\psi \in X:u_{\psi }\in W^{2,p}\left( a,b;X\right)
\cap L^{p}\left( a,b;D(M^{2})\right)$. 

\noindent2. $u_{\psi }\in W^{m,p}\left( a,b;X\right) \cap L^{p}\left(
a,b;D(M^{m})\right) \Longleftrightarrow \psi \in \left( D\left(
M^{m-2}\right) ,X\right) _{\frac{1}{\left( m-2\right) p},p}$.
 
In this case, for all $\ell \in \left\{ 1,...,m-1\right\}$
\begin{equation*}
u_{\psi }^{\left( \ell \right) }\in W^{m-\ell ,p}\left( a,b;X\right) \cap
L^{p}\left( a,b;D(M^{m-\ell })\right).
\end{equation*}

In particular,
\begin{equation*}
u_{\psi }\in W^{4,p}\left( a,b;X\right) \cap L^{p}\left( a,b;D(M^{4})\right) \Longleftrightarrow \psi \in \left( D\left( M\right) ,X\right) _{1+\frac{1}{p},p},
\end{equation*}

and in this case $u_{\psi }^{\prime \prime }\in L^{p}\left(a,b;D(M^{2})\right)$.

\noindent3. The previous statement holds true if we replace $u_\psi$ by $v_\psi$.
\end{Cor}

\begin{proof}
We set $\B=B\left( L+M\right) ^{-1}$. Moreover, $\B$ satisfies 
\begin{equation*}
\left\{\begin{array}{l}
1)~\B(X) \subset D(L+M)=D\left( M\right) \text{ then }M\B \in \mathcal{L}\left( X\right) . \\ 
2)~0\in \rho \left( M\B\right) \text{ with }\left( M\B\right) ^{-1}=\left(L+M\right) M^{-1}B^{-1}.
\end{array}\right.
\end{equation*}
From \eqref{(L-M)y=}, we have $L=M+B\left( L+M\right)^{-1}$. Then, setting $\Lambda_2=L$, $\Lambda_1=M$ and $\B = B\left( L+M\right)^{-1}$ in statement $6.$ of \refT{Reg Diff exp}, we obtain statements $1.$ and $2.$ Moreover, from statement $7.$ of \refT{Reg Diff exp}, we obtain statement~$3.$
\end{proof}

\subsection{Proof of \refT{thmprinc gen}} \label{sect th pr gen}

We first give a useful remark concerning the regularity.
\begin{Rem}\label{Rem trace}
From \refL{Lem trace}, if $u \in W^{4,p}(a,b;X) \cap L^p(a,b;D(M^4))$ then, for $s \in [a,b]$, we obtain
$$
u(s)\in (D(M^4),X)_{\frac{1}{4p},p},\quad u'(s)\in (D(M^4),X)_{\frac{1}{4}+\frac{1}{4p},p}\quad \text{and}\quad u''(s)\in (D(M^4),X)_{\frac{1}{2}+\frac{1}{4p},p}.
$$
Moreover, from \refR{Rem Réitération}, for $s \in [a,b]$, we deduce that
$$
u(s)\in (D(M),X)_{3+\frac{1}{p},p},\quad u'(s)\in (D(M),X)_{2+\frac{1}{p},p}\quad \text{and}\quad u''(s)\in (D(M),X)_{1+\frac{1}{p},p}.
$$
In the same way, we obtain the following equalities:
\begin{equation}\label{egalités espaces interpol P}
\left\{\begin{array}{l}
(D(M),X)_{3+\frac{1}{p},p} = (D(P),X)_{1+\frac{1}{2p},p}, \\ \ecart
(D(M),X)_{2+\frac{1}{p},p} = (D(P),X)_{1 + \frac{1}{2} + \frac{1}{2p},p}, \\ \ecart
(D(M),X)_{1+\frac{1}{p},p} = (D(P),X)_{\frac{1}{2p},p}.
\end{array}\right.
\end{equation}
\end{Rem}

Thus, we only have to prove the converse implications in \refT{thmprinc gen}. 

\subsubsection{Proof of 1. of \refT{thmprinc gen} (Boundary Conditions (BC2))}

Assume that $(H_1)$, $(H_2)$, $(H_3)$, $(H_4)$, $(H_5)$ and \eqref{Reg Data 2} hold. 

If $u$ is a classical solution of \eqref{eq de base}-\eqref{condbord2}, then from \refP{Proposition 1}, $u$ satisfies \eqref{Repre u}. We set
\begin{equation}\label{al gen}
\alpha_1 = \frac{K_1-K_2}{2},\quad \alpha_2 = \frac{K_3-K_4}{2},\quad \alpha_3 =\frac{K_1+K_2}{2}\quad\text{and}\quad \alpha_4 = \frac{K_3+K_4}{2}.
\end{equation}
Then, for $a.\,e.~x \in (a,b)$, $u$ is given by
\begin{equation}\label{Repre u bis}
\begin{array}{rcl}
u(x)&=&\dis~~ \,\left(e^{(x-a) M}-e^{(b-x) M}\right)\alpha _{1}+\left(e^{(x-a) L}-e^{(b-x) L}\right)\alpha_{2}
\\\ecart
&&\dis+\left(e^{(x-a) M}+e^{(b-x) M}\right)\alpha_{3}+\left(e^{(x-a) L}+e^{(b-x) L}\right)\alpha_4 + F_{0 ,f}\left( x\right).
\end{array}
\end{equation}
Following the same steps as those used in the proof of Theorem 2.5, p. 365 in \cite{LMMT} (where we replace $k I$ by $B$), we obtain
\begin{equation} \label{ALPHA 1 2 3 4}
\left\{ \begin{array}{rcl}
\alpha_1 & = & \left(I + e^{c M}\right)^{-1} \left( M^{-1} \tilde{\varphi_1} - \left(I + e^{cL}\right) LM^{-1} \alpha_2 \right) \\ \ecart 
\alpha_2 & = & \dis \left(I - e^{c L}\right)^{-1} B^{-1} \left(\frac{\varphi_3 - \varphi_4}{2}\right) \\ \ecart 
\alpha_3 & = & \left(I - e^{c M}\right)^{-1} \left(M^{-1} \tilde{\varphi_2} - \left(I - e^{cL}\right) LM^{-1} \alpha_4 \right) \\ \ecart
\alpha_4 & = & \dis \left(I + e^{c L}\right)^{-1} B^{-1} \left(\frac{\varphi_3 + \varphi_4}{2}\right),
\end{array}\right.  
\end{equation}
where
\begin{equation}\label{PHI 1 2 tilde}
\tilde{\varphi_1}\; := \; \frac{\varphi_1 + \varphi_2 - F'_{0,f}(a) - F'_{0,f}(b)}{2} \quad \text{and} \quad \tilde{\varphi_2} \; := \; \frac{\varphi_1 - \varphi_2 - F'_{0,f}(a) + F'_{0,f}(b)}{2}.
\end{equation}
Now, thanks to \refL{Lem triebel}, \refL{Lem reg}, \refC{Cor diff exp} and using again the same method as in \cite{LMMT}, we obtain that $u$ is the unique classical solution of \eqref{eq de base}-\eqref{condbord2}.

\subsubsection{Proof of 2. of \refT{thmprinc gen} (Boundary Conditions (BC3))}

Assume that $(H_1)$, $(H_2)$, $(H_3)$, $(H_4)$, $(H_5)$, $(H_6)$ and \eqref{Reg Data 3} hold. 

As previously, following the same steps as those used in the proof of Theorem 2.5, p. 365 in \cite{LMMT} (where we replace $k I$ by $B$) and using \refL{Lem triebel}, \refL{Lem reg} and \refC{Cor diff exp}, we deduce that $u$ is the unique classical solution of \eqref{eq de base}-\eqref{condbord3}, given by \eqref{Repre u} where  
\begin{equation}\label{ali gen}
\left\{ \begin{array}{rrc}
\dis \alpha_1& =&\dis\frac{1}{2}B^{-1}(L+M)U^{-1}\left[L(I+e^{cL})(\varphi_1-\varphi_2)-2(I-e^{cL})\tilde{\varphi}_1\right]\\\ecart
\dis \alpha_2& =&\dis-\frac{1}{2}B^{-1}(L+M)U^{-1}\left[M(I+e^{cM})(\varphi_1-\varphi_2)-2(I-e^{cM})\tilde{\varphi}_1\right]\\\ecart
\alpha_3& =&\dis\frac{1}{2}B^{-1}(L+M)V^{-1}\left[L(I-e^{cL})(\varphi_1+\varphi_2)-2(I+e^{cL})\tilde{\varphi}_2\right]\\\ecart
\alpha_4& =&\dis-\frac{1}{2}B^{-1}(L+M)V^{-1}\left[M(I-e^{cM})(\varphi_1+\varphi_2)-2(I+e^{cM})\tilde{\varphi}_2\right],
\end{array}\right.
\end{equation}
with \begin{equation}\label{tildephi gen}
\tilde{\varphi}_1\;:=\;\frac{\varphi_3+\varphi_4-F_{0,f}'(a)-F_{0,f}'(b)}{2}\quad\text{and}\quad \tilde{\varphi}_2\;:=\;\frac{\varphi_3-\varphi_4-F_{0,f}'(a)+F_{0,f}'(b)}{2}.
\end{equation}

\subsubsection{Proof of 3. of \refT{thmprinc gen} (Boundary Conditions (BC4))}

Assume that $(H_1)$, $(H_2)$, $(H_3)$, $(H_4)$, $(H_5)$, $(H_6)$ and \eqref{Reg Data 4} hold. We proceed as in the previous proof. 

Following the same steps as those used in the proof of Theorem 2.5, p. 365 in \cite{LMMT} (where we replace $k I$ by $B$) and using \refL{Lem triebel}, \refL{Lem reg} and \refC{Cor diff exp}, we deduce that $u$ is the unique classical solution of \eqref{eq de base}-\eqref{condbord4}, given by \eqref{Repre u} with  
\begin{equation}\label{ali2 gen}
\left\{\begin{array}{rcl}
\alpha_1 & =&\dis\frac{1}{2} B^{-1}(L+M)V^{-1}\left[2(I-e^{c L})LM^{-1}\tilde{\varphi}_1-(I+e^{c L})M^{-1}(\varphi_3-\varphi_4)\right]\\\ecart
\alpha_2 & =&\dis-\frac{1}{2} B^{-1}(L+M)V^{-1}\left[2(I-e^{c M})ML^{-1}\tilde{\varphi}_1- (I+e^{c M}) L^{-1}(\varphi_3-\varphi_4)\right]\\\ecart
\alpha_3 & =&\dis\frac{1}{2} B^{-1}(L+M)U^{-1}\left[2(I+e^{c L})LM^{-1}\tilde{\varphi}_2-(I-e^{c L})M^{-1}(\varphi_3+\varphi_4)\right]\\\ecart
\alpha_4 & =&\dis-\frac{1}{2} B^{-1}(L+M)U^{-1}\left[2(I+e^{c M})ML^{-1}\tilde{\varphi}_2 -(I-e^{c M})L^{-1}(\varphi_3+\varphi_4)\right],
\end{array}\right.
\end{equation}
where $\tilde{\varphi_1}$ and $\tilde{\varphi_2}$ are given by \eqref{PHI 1 2 tilde}.

\section{Back to the parabolic problem}\label{sect linear parabolic case}

Let $X$ be a complex Banach space and $\A_i$, $i=1,2,3,4,5$, the linear operator defined by
\begin{equation}\label{Ai}
\left\{\begin{array}{ccl}
D(\A_i) & = & \{u \in W^{4,p}(a,b;X)\cap L^p(a,b;D(A^2))\text{ and } u'' \in L^p(a,b;D(A)) : \text{(BCi)}_0\} \\ \ecart
\left[\A_i u\right] (x) & = & -u^{(4)}(x) - (2A - kI) u''(x) - (A^2 - kA) u(x), \quad u \in D(\A_i), ~x\in (a,b),
\end{array}\right.
\end{equation}
where $k \in \RR$ and (BCi)$_0$ represents the boundary conditions (BCi), $i=1,2,3,4,5$, with $\varphi_j = 0$, $j=1,2,3,4$ and wherein $A_0$ is replaced by a more general closed linear operator $A$ satisfying the assumptions below. 

Then, we will study the spectral properties of $\A_i$ in order to solve the Cauchy problem \eqref{P parabo}, where $f$ and $v_0$ are in appropriate spaces. 

\subsection{Assumptions and main results}

\subsubsection{Assumptions}

Let $A$ be a closed linear operator and assume
\begin{center}
\begin{tabular}{l}
$(\H_1)\quad X$ is a UMD space, \\ \ecart
$(\H_2)\quad 0 \in \rho(A)$, \\ \ecart
$(\H_3)\quad -A \in$ BIP$(X,\theta_A),~$ for $\theta_A \in [0, \pi/2)$,\\ \ecart
$(\H_4)\quad [k,+\infty) \in \rho(A)$.
\end{tabular}
\end{center}
In some results, for the boundary conditions \eqref{condbord3} and \eqref{condbord4}, we may need a supplementary hypothesis:
$$(\H_5)\quad -A \in \text{Sect}(0).$$

\subsubsection{Main results}

\begin{Prop}\label{Prop Sect}
Assume that $(\H_1)$, $(\H_2)$, $(\H_3)$ and $(\H_4)$ hold. Then
\begin{enumerate}
\item for $i=1,2,5$, we have
$$-\A_i + \frac{k^2}{4} I \in \text{Sect}\,(2\theta_A),$$

\item for $i=3,4$, we have
$$\left\{\begin{array}{ll}
\dis -\A_i + \frac{k^2}{4} I + r I \in \text{Sect}\,\left(\frac{\pi}{2}\right), & \text{if } \dis 2\theta_A \in \left[0, \frac{\pi}{2}\right) \\ \ecart
\dis -\A_i + \frac{k^2}{4} I + r I \in \text{Sect}\,\left(2\theta_A\right), & \text{if } \dis 2\theta_A \in \left[\frac{\pi}{2}, \pi\right),
\end{array}\right.$$
where $r > 0$ is defined in \refP{Prop rho(-Ai)2}. 

Moreover, there exist $r' > r$ and $\theta_0 > 0$, such that
$$\left\{\begin{array}{ll}
\dis -\A_i + \frac{k^2}{4} I + r' I \in \text{Sect}\,\left(2\theta_A\right), & \text{if } \dis 2\theta_A \in \left(0, \frac{\pi}{2}\right) \\ \ecart

\dis -\A_i + \frac{k^2}{4} I + r' I \in \text{Sect}\,\left(\theta_0\right), & \text{if } \dis 2\theta_A = 0.
\end{array}\right.$$
\end{enumerate}
\end{Prop}

\begin{Th}\label{Th final}
Assume that $(\H_1)$, $(\H_2)$, $(\H_3)$ and $(\H_4)$ hold. Then, for $i=1,2,3,4,5$, if $\theta_A < \pi/4$, operator $\A_i$ is the infinitesimal generator of an analytic strongly continuous semigroup $\left(e^{t \A_i}\right)_{t \geqslant 0}$.
\end{Th}

\begin{Rem}
For $i = 1,2,5$ and $\theta_A < \pi/4$, from \refP{Prop Sect}, we have
$$\exists \, M_i \geqslant 1 : \forall\,t \geqslant 0, \quad \left\|e^{t \A_i}\right\|_{\L(X)} \leqslant M_i \,e^{t \frac{k^2}{4}}.$$
\end{Rem}

First of all, we have to study the spectral properties of $\A_i$. Thus, we focus on the resolvent set of $-\A_i$. To this end, we analyse the equation $(-\A_i- \lambda I)u=f$.

\subsection{Study of the resolvent set}\label{sect ens resolv}

Let $\lambda \in \CC$ and fix $i \in \{1,2,3,4,5\}$. By definition $\lambda \in \rho(\A_i)$ means that the following equation
\begin{equation}\label{EDA}
u^{(4)}(x) + (2A - kI) u''(x) + (A^2 - kA - \lambda I) u(x) = f(x), \quad x \in (a,b),
\end{equation}
supplemented by the boundary conditions (BCi)$_0$ admits a unique solution $u$ in $D(\A_i)$. Since $A$ satisfies the first four previous assumptions, we define
$$A_{k/2} = A - \frac{k}{2}I.$$
Let $\lambda \in \CC \setminus (-k^2/4,+\infty)$ which means that $\dis -\lambda - \frac{k^2}{4} \in \CC \setminus (-\infty,0)$. We set
\begin{equation}\label{PQ}
P_\lambda = A_{k/2} + i\sqrt{-\lambda - \frac{k^2}{4}}I \quad \text{and} \quad Q_\lambda = A_{k/2} - i\sqrt{-\lambda - \frac{k^2}{4}}I,
\end{equation}
so that we can rewrite equation \eqref{EDA} as
\begin{equation*}
u^{(4)}(x) + (P_\lambda + Q_\lambda) u''(x) + P_\lambda Q_\lambda u(x) = f(x), \quad x \in (a,b),
\end{equation*}
and use results of \refs{sect linear steady case}.

We state three technical lemmas which allow us to justify the other results of this section. We first recall that if $z \in \CC$, then $\arg(z)$ is the unique argument of $z$ in $(-\pi,\pi]$.

\begin{Lem}\label{Lem sect}
Let $\lambda \in \CC$. The two following statements are equivalent
\begin{itemize}
\item $\dis\lambda \in -\frac{k^2}{4} + \left(\CC \setminus \overline{S_{2\theta_A}}\right),$

\item $\dis\lambda \neq -\frac{k^2}{4}$ and $\dis\left|\arg\left(-\lambda-\frac{k^2}{4}\right) \pm \pi\right| < 2\left(\pi- \theta_A\right).$
\end{itemize}
\end{Lem}

\begin{proof}
Let $\lambda \in \CC$, then
$$\left|\arg\left(-\lambda-\frac{k^2}{4}\right) \pm \pi\right| < 2\left(\pi- \theta_A\right),$$
is equivalent to
\begin{equation}\label{syst arg}
\dis\arg\left(-\lambda-\frac{k^2}{4}\right)   < \pi - 2 \theta_A \quad \text{and} \quad \dis - \arg\left(-\lambda-\frac{k^2}{4}\right) < \pi - 2 \theta_A.
\end{equation}
Now, it remains to study the two following cases.
\begin{itemize}
\item First case : $\dis \arg\left(-\lambda-\frac{k^2}{4}\right) \geqslant 0$.

Here, \eqref{syst arg} is equivalent to 
$$0 \leqslant \arg\left(-\lambda-\frac{k^2}{4}\right) < \pi - 2 \theta_A,$$
and using $\dis\arg\left(\lambda+\frac{k^2}{4}\right) = \arg\left(-\lambda-\frac{k^2}{4}\right) - \pi$, then \eqref{syst arg} becomes
$$-\pi \leqslant \arg\left(\lambda + \frac{k^2}{4}\right) < -2 \theta_A.$$

\item Second case : $\dis\arg\left(-\lambda-\frac{k^2}{4}\right) < 0$. 

Now, \eqref{syst arg} writes
$$0 \leqslant -\arg\left(-\lambda-\frac{k^2}{4}\right) < \pi - 2 \theta_A,$$
and using $\dis\arg\left(\lambda+\frac{k^2}{4}\right) = \arg\left(-\lambda-\frac{k^2}{4}\right) + \pi$, then \eqref{syst arg} becomes
$$2 \theta_A < \arg\left(\lambda + \frac{k^2}{4}\right) \leqslant \pi.$$
\end{itemize}
Finally, \eqref{syst arg} is equivalent to $\lambda + \dis \frac{k^2}{4} \in \CC \setminus \overline{S_{2\theta_A}}$, which gives the result.
\end{proof}

\begin{Lem}\label{Lem PQ BIP} 
Assume that $(\H_1)$, $(\H_2)$, $(\H_3)$ and $(\H_4)$ hold. Let $\lambda \in - \dis\frac{k^2}{4} + \left(\CC \setminus \overline{S_{2\theta_A}}\right)$. Then, we have
\begin{enumerate}
\item $\dis \left|\arg\left(\pm i\sqrt{-\lambda - \frac{k^2}{4}}\right) \right| = \frac{\left|\arg\left(-\lambda-\frac{k^2}{4}\right) \pm \pi \right|}{2} < \pi - \theta_A.$

\item $-P_\lambda \in \text{BIP}\,(X,\theta_1)$ and $-Q_\lambda \in \text{BIP}\,(X,\theta_2)$, with $\theta_1,\theta_2 \in [0,\pi-\theta_A)$, where
$$\theta_1 := \max\left(\theta_A, \frac{\left|\arg\left(-\lambda-\frac{k^2}{4}\right) + \pi\right|}{2}\right) \quad\text{and} \quad \theta_2 := \max\left(\theta_A, \frac{\left|\arg\left(-\lambda-\frac{k^2}{4}\right) - \pi \right|}{2}\right).$$ 
\end{enumerate}
\end{Lem}
\begin{proof}\hfill
\begin{enumerate}
\item Since $\theta_A < \pi/2$, the result follows from \refL{Lem sect}.

\item From $(\H_1)$, $(\H_2)$, $(\H_3)$, $(\H_4)$ and \cite{arendt-bu-haase}, Theorem 2.3, p. 69, we have
$$-A_{k/2} \in \text{BIP}(X,\theta_A),$$
hence, if $\lambda = -k^2/4$, we obtain
$$-P_{-k^2/4} = -Q_{-k^2/4}  = -A_{k/2} \in \text{BIP}\,(X,\theta_A).$$
Moreover, if $\lambda \in - \dis\frac{k^2}{4} + \left(\CC \setminus \overline{S_{2\theta_A}}\right)$, then 
\begin{equation}\label{égalité arg}
\left|\arg\left(\pm i\sqrt{-\lambda - \frac{k^2}{4}}\right)\right| = \left|\pm \frac{\pi}{2} + \frac{\arg\left(-\lambda-\frac{k^2}{4}\right)}{2}\right| \in (0,\pi),
\end{equation}
so $\pm i\sqrt{-\lambda - \frac{k^2}{4}} \notin \RR$. 

Finally, $-A_{k/2} \in$ BIP$\,(X,\theta_A)$ and $\pm i\sqrt{-\lambda - \frac{k^2}{4}} \in \CC\setminus (-\infty,0)$. Moreover, from  \eqref{égalité arg} and \refL{Lem sect}, we have
\begin{equation}\label{condition parabolicité}
\theta_A + \left|\arg\left(\pm i\sqrt{-\lambda - \frac{k^2}{4}}\right)\right| < \pi,
\end{equation}
thus, from \cite{monniaux}, Theorem 2.4, p. 408, we obtain
$$-P_\lambda \in \text{BIP}(X,\theta_1)\quad \text{and} \quad -Q_\lambda \in \text{BIP}(X,\theta_2).$$
\end{enumerate}
\end{proof}
Now, in order to use \cite{dore-yakubov} in the next proof, we first need to give the following remark.
\begin{Rem}\label{Rem type phi}
In the sequel, we will use results of \cite{dore-yakubov} in which the authors use operators of type $\varphi$ instead of sectorial operators. For $\varphi \in (0,\pi)$. A closed  linear operator $T:D(T) \subset X\longrightarrow X$ is said of type $\varphi $ with bound $C$ if and only if $\overline{S_{\varphi}}\subset \rho(-T) $ and
\begin{equation*}
\forall \lambda \in \overline{S_{\varphi}},\quad\left \Vert \left(T + \lambda I\right)
^{-1}\right \Vert _{\mathcal{L}(X)}\leqslant \frac{C}{1+\left \vert \lambda
\right \vert }.
\end{equation*}
It is clear that if $T$ is of type $\varphi$, then $T \in$ Sect$(\pi - \varphi)$. More precisely, the two notions are linked by the equivalence of the two following assertions:
\begin{enumerate}
\item $T \in$ Sect$(\theta_T)$ with $\theta_T \in [0,\pi)$ and $0 \in \rho(T)$,
\item $\forall ~\varepsilon \in (0, \pi- \theta_T)$, $T$ is of type $\varphi = \pi-\theta_T -\varepsilon \in (0,\pi)$.
\end{enumerate}   
\end{Rem}
\begin{Lem}\label{Lem LM^-1}
Let $k \in \RR$ and $\lambda \in -k^2/4 + \left(\CC \setminus \overline{S_{2\theta_A}}\right)$ and assume that $(\H_1)$, $(\H_2)$, $(\H_3)$ and $(\H_4)$ hold. Then $0\in \rho(P_\lambda) \cap \rho(Q_\lambda)$ and there exists $C > 0$, independent of $\lambda$, such that
$$\left\|M_\lambda L_\lambda^{-1}\right\|_{\L(X)} \leqslant C \quad \text{and} \quad \left\|L_\lambda M_\lambda^{-1}\right\|_{\L(X)} \leqslant C,$$
where $M_\lambda := -\sqrt{-P_\lambda}$ and $L_\lambda := -\sqrt{-Q_\lambda}$.
\end{Lem}
\begin{proof}
Since $-A_{k/2} \in$ BIP\,$(X,\theta_A)$, its spectral angle $\psi_{A,k}$ satisfies $0\leqslant \psi_{A,k} \leqslant \theta_A$ (see \cite{pruss}, p.~218). Moreover, $\psi_{\lambda,k}$, the spectral angle of $i\sqrt{-\lambda - \frac{k^2}{4}}\,I$, satisfies $\psi_{\lambda,k} = \left|\arg\left(i\sqrt{-\lambda - \frac{k^2}{4}}\right)\right|$. Furthermore, from \eqref{condition parabolicité}, it follows that
$$\psi_{A,k} + \psi_{\lambda,k} \leqslant \theta_A + \left|\arg\left(i\sqrt{-\lambda - \frac{k^2}{4}}\right)\right| < \pi.$$
Then due to Theorem 8.3 (iv), p. 218 in \cite{pruss}, we have $0 \in \rho(P_\lambda)$. In the same way, we obtain $0 \in \rho(Q_\lambda)$. Thus, from \refL{Lem PQ BIP}, we deduce $\sqrt{-P_\lambda}$ and $\sqrt{-Q_\lambda}$ are well defined and invertible with bounded inverse. Since we have
$$M_\lambda L_\lambda^{-1} = M_\lambda^2 M_\lambda^{-1}L_\lambda^{-1} = -P_\lambda M_\lambda^{-1} L_\lambda^{-1},$$
we deduce that
$$\begin{array}{lll}
M_\lambda L_\lambda^{-1} & = & \dis \left(- A_{k/2} - i \sqrt{-\lambda-\frac{k^2}{4}}\, I\right) M_\lambda^{-1} L_\lambda^{-1} \\ \ecart

& = & \dis -A_{k/2} M_\lambda^{-1} L_\lambda^{-1} - i \sqrt{-\lambda-\frac{k^2}{4}}\,  M_\lambda^{-1} L_\lambda^{-1} \\ \ecart

& = & \dis \sqrt{-A_{k/2}} \, M_\lambda^{-1} \sqrt{-A_{k/2}} \, L_\lambda^{-1} - i \sqrt{-\lambda-\frac{k^2}{4}}\,  M_\lambda^{-1} L_\lambda^{-1} \\ \ecart

& = & \dis \sqrt{-A_{k/2}}  \left(- A_{k/2} - i \sqrt{-\lambda-\frac{k^2}{4}}\, I\right)^{-\frac{1}{2}}  \sqrt{-A_{k/2}}  \left(- A_{k/2} + i \sqrt{-\lambda-\frac{k^2}{4}}\, I\right)^{-\frac{1}{2}} \\ \ecart
&&\dis - i \sqrt{-\lambda-\frac{k^2}{4}}\,  \left(- A_{k/2} - i \sqrt{-\lambda-\frac{k^2}{4}}\, I\right)^{-\frac{1}{2}}  \left(- A_{k/2} + i \sqrt{-\lambda-\frac{k^2}{4}}\, I\right)^{-\frac{1}{2}}. 
\end{array}$$
Thus, it follows that
$$\begin{array}{lll}
\left\|M_\lambda L_\lambda^{-1} \right\|_{\L(X)} & \leqslant & \dis \left\|\sqrt{-A_{k/2}}  \left(- A_{k/2} - i \sqrt{-\lambda-\frac{k^2}{4}}\, I\right)^{-\frac{1}{2}}  \right\|_{\L(X)} \\ \ecart
&& \dis \times \left\| \sqrt{-A_{k/2}} \left(- A_{k/2} + i \sqrt{-\lambda-\frac{k^2}{4}}\, I\right)^{-\frac{1}{2}} \right\|_{\L(X)} \\ \ecart
&& \dis + \sqrt{\left|\lambda + \frac{k^2}{4}\right|} \left\| \left(- A_{k/2} - i \sqrt{-\lambda-\frac{k^2}{4}}\, I\right)^{-\frac{1}{2}} \right\|_{\L(X)}\\ \ecart
&& \dis \times \left\|\left(- A_{k/2} + i \sqrt{-\lambda-\frac{k^2}{4}}\, I\right)^{-\frac{1}{2}}\right\|_{\L(X)}. 
\end{array}$$
Thus, due to $(\H_2)$ and \refR{Rem type phi}, we can apply \cite{dore-yakubov}, Lemma 2.6 statement a), p. 104: there exist $C_1,C_2 > 0$, which are independent of $\lambda$, such that
\begin{equation}\label{ineg C1} 
\left\|\sqrt{-A_{k/2}} \left(- A_{k/2} \pm i \sqrt{-\lambda-\frac{k^2}{4}}\, I\right)^{-\frac{1}{2}} \right\|  \leqslant C_1,
\end{equation}
and
\begin{equation}\label{ineg C2}
\left\| \left(- A_{k/2} \pm i \sqrt{-\lambda-\frac{k^2}{4}}\, I\right)^{-\frac{1}{2}} \right\|_{\L(X)} \leqslant \frac{C_2}{\sqrt{\sqrt{\left|\lambda + \frac{k^2}{4}\right|}}}.
\end{equation}
Therefore, \eqref{ineg C1} and \eqref{ineg C2} give us the expected estimate for $\|M_\lambda L_\lambda^{-1}\|$. In the same way, replacing $M_\lambda$ by $L_\lambda$ and $L_\lambda$ by $M_\lambda$, we obtain the symmetric result for $\|L_\lambda
M_\lambda^{-1}\|$.
\end{proof}
Now, we adapt a useful technical lemma from \cite{dore-yakubov}.
\begin{Lem}[\cite{dore-yakubov}]\label{Lem DY}
Let $\dis\lambda \in -\frac{k^2}{4} + \left(\CC \setminus \overline{S_{2\theta_A}}\right)$ and assume that $(\H_1)$, $(\H_2)$, $(\H_3)$ and $(\H_4)$ hold. Then, the two analytic semigroups $\left(e^{-t \sqrt{-P_\lambda}}\right)_{t\geqslant 0}$ and $\left(e^{-t \sqrt{-Q_\lambda}}\right)_{t \geqslant 0}$, are well defined. 

Moreover, let $\alpha \in \RR$ and fix $t_0 > 0$, there exist $K > 0$ and $\tilde{\omega} > 0$, such that for any $\dis\lambda \in -\frac{k^2}{4} + \left(\CC \setminus \overline{S_{2\theta_A}}\right)$, we have 
$$\left\|\left(-P_\lambda\right)^\alpha e^{-t_0 \sqrt{-P_\lambda}}\right\|_{\L(X)} \leqslant K e^{-t_0 \tilde{\omega} \left|\lambda + \frac{k^2}{4}\right|^{1/4}} \quad \text{and} \quad \left\|\left(-Q_\lambda\right)^\alpha e^{-t_0 \sqrt{-Q_\lambda}}\right\|_{\L(X)} \leqslant K e^{-t_0 \tilde{\omega} \left|\lambda + \frac{k^2}{4}\right|^{1/4}}.$$ 
\end{Lem}

\begin{proof}
Let $\dis\lambda \in -\frac{k^2}{4} + \left(\CC \setminus \overline{S_{2\theta_A}}\right)$. From \refL{Lem PQ BIP} and \cite{haase}, Proposition 3.1.2, p. 63, we deduce that
$$\sqrt{-P_\lambda} \in \text{BIP}\,(X,\theta_1/2) \quad \text{and} \quad \sqrt{-Q_\lambda} \in \text{BIP}\,(X,\theta_2/2),$$
with $\theta_1/2,\theta_2/2 \in [0, \pi/2)$. Then, we deduce that $-\sqrt{-P_\lambda}$ and $-\sqrt{-Q_\lambda}$ are the infinitesimal generators of the two analytic semigroups 
$$\left(e^{-t \sqrt{-P_\lambda}}\right)_{t\geqslant 0} = \left(e^{-t \sqrt{-A_{k/2} - i\sqrt{-\lambda-\frac{k^2}{4}}I}}\right)_{t\geqslant 0},$$
and
$$\left(e^{-t \sqrt{-Q_\lambda}}\right)_{t\geqslant 0} = \left(e^{-t \sqrt{-A_{k/2} + i \sqrt{-\lambda-\frac{k^2}{4}}I}}\right)_{t \geqslant 0}.$$
Moreover, let $\alpha \in \RR$ and fix $t_0>0$, from \cite{dore-yakubov}, Lemma~2.6, statement b), p. 104, we deduce that there exist $K \geqslant 1$ and $\tilde{\omega} > 0$, such that 
$$\begin{array}{lll}
\dis \left\|(-P_\lambda)^\alpha e^{-t_0\sqrt{-P_\lambda}} \right\|_{\L(X)} & = & \dis \left\|\left(- A_{k/2} - i \sqrt{-\lambda-\frac{k^2}{4}}I\right)^\alpha e^{-t_0 \sqrt{-A_{k/2} - i \sqrt{-\lambda-\frac{k^2}{4}}I}}\right\|_{\L(X)} \\ \ecart
& \leqslant & \dis K e^{-t_0 \tilde{\omega} \sqrt{\left|i\sqrt{-\lambda-\frac{k^2}{4}}\right|}},
\end{array}$$
and
$$\begin{array}{lll}
\dis \left\|(-Q_\lambda)^\alpha e^{-t_0\sqrt{-Q_\lambda}} \right\|_{\L(X)} & = & \dis \left\|\left(- A_{k/2} + i \sqrt{-\lambda-\frac{k^2}{4}}I\right)^\alpha e^{-t_0 \sqrt{-A_{k/2} + i \sqrt{-\lambda-\frac{k^2}{4}}I}}\right\|_{\L(X)} \\ \ecart
& \leqslant & \dis K e^{-t_0 \tilde{\omega} \sqrt{\left|i\sqrt{-\lambda-\frac{k^2}{4}}\right|}},
\end{array}$$
with
$$K e^{-t_0 \tilde{\omega} \sqrt{\left|i\sqrt{-\lambda-\frac{k^2}{4}}\right|}} = K e^{-t_0 \tilde{\omega} \sqrt{\sqrt{\left|\lambda + \frac{k^2}{4}\right|}}} = K e^{-t_0 \tilde{\omega} \left|\lambda + \frac{k^2}{4}\right|^{1/4}}.$$
\end{proof}
\begin{Prop}\label{Prop rho(-Ai)}
Let $i=1,2,5$ and assume that $(\H_1)$, $(\H_2)$, $(\H_3)$, $(\H_4)$ hold. Then, we have
$$\{0\} \cup \left\{-\frac{k^2}{4}\right\} \cup \left(-\frac{k^2}{4} + \left(\CC \setminus \overline{S_{2\theta_A}}\right)\right) \subset \rho(-\A_i).$$
\end{Prop}

\begin{proof}\hfill
\begin{itemize}
\item If $\lambda=0$, we obtain, for $k\neq 0$, that $0\in \rho(\A_i)$, from \cite{LMMT}, Theorem 2.2, p. 355 and Theorem~2.5, p. 356-357, by taking 
$$\left\{\begin{array}{llll}
P_\lambda=A-kI & \text{and} & Q_\lambda=A,    & \text{ if }k > 0, \\ \ecart
P_\lambda=A    & \text{and} & Q_\lambda=A-kI, & \text{ if }k < 0,
\end{array}\right.$$  
Moreover, if $k=0$, from \cite{thorel}, Theorem~2.6 and Theorem~2.8, then $0\in \rho(\A_i)$.

\item If $\dis\lambda = -\frac{k^2}{4}$, then from \refL{Lem PQ BIP}, $-P_{-k^2/4} = -Q_{-k^2/4} = -A_{k/2} \in$ BIP\,$(X,\theta)$, where $\theta = \max(\theta_1,\theta_2)$. Thus, from \cite{thorel}, Theorem~2.6 and Theorem~2.8, we have $-k^2/4 \in \rho(-\A_i)$.

\item Let $\dis\lambda \in -\frac{k^2}{4} + \left(\CC \setminus \overline{S_{2\theta_A}}\right)$. Then, using $P_\lambda$ and $Q_\lambda$ defined by \eqref{PQ}, we obtain that 
\begin{equation}\label{B lambda}
B_\lambda := 2i\sqrt{-\lambda-\frac{k^2}{4}} I \in \L(X),
\end{equation}
is invertible with bounded inverse. Moreover, we have $P_\lambda = Q_\lambda + B_\lambda$. 

We are now in position to apply the results of \refs{sect linear steady case} with $P$, $Q$ replaced by $P_\lambda$, $Q_\lambda$. From \refL{Lem PQ BIP}, it follows that assumptions $(H_1)$, $(H_2)$, $(H_3)$, $(H_4)$ and $(H_5)$ of \refs{sect hyp gen} are satisfied. Then, from \refT{Th Carre} and \refT{thmprinc gen}, there exists a unique classical solution of \eqref{EDA}-(BCi)$_0$. Thus, we deduce that have $\lambda \in \rho(-\A_i)$. 
\end{itemize}
\end{proof}

\begin{Prop}\label{Prop rho(-Ai)2}
Let $i=3,4$ and assume that $(\H_1)$, $(\H_2)$, $(\H_3)$, $(\H_4)$ hold. Then, there exists $r > 0$, such that
$$\left\{-\frac{k^2}{4}\right\} \cup \left(-\frac{k^2}{4} + \left(\CC \setminus \left( \overline{B\left(0,r\right)} \cup \overline{S_{2\theta_A}}\right)\right) \right) \subset \rho(-\A_i).$$
Moreover, if in addition, we assume $(\H_5)$, we obtain $\dis 0 \in \rho(-\A_i)$.
\end{Prop}

\begin{proof}
As in the proof of \refP{Prop rho(-Ai)}, assumptions $(H_1)$, $(H_2)$, $(H_3)$, $(H_4)$ and $(H_5)$ of \refs{sect hyp gen} are satisfied and also $\dis -\frac{k^2}{4} \in \rho(-\A_i)$. 

Now, we adapt notations of \refs{sect hyp gen}, replacing $L$, $M$, $U$, $V$, $T^-$, $T^+$, $B$ and $M_B$ by $L_\lambda$, $M_\lambda$, $U_\lambda$, $V_\lambda$, $T_{\lambda}^-$, $T_{\lambda}^+$, $B_\lambda$ and $M_{B_\lambda}$. Our aim is to show that $(H_6)$ of \refs{sect hyp gen} holds with $U_\lambda$ and $V_\lambda$ instead of $U$ and $V$. To this end, we recall that $c = b-a > 0$. From \refL{Lem DY}, for $t_0 = c > 0$, there exist $K \geqslant 1$ and some $\tilde{\omega} > 0$, such that, for any $\dis\lambda \in -\frac{k^2}{4} + \left(\CC \setminus \overline{S_{2\theta_A}}\right)$, we have
$$\left\{\begin{array}{lllll}
\dis\|M_\lambda^2 e^{cM_\lambda}\|_{\L(X)} &=&\dis \|-P_\lambda e^{-c\sqrt{-P_\lambda}}\|_{\L(X)} & \leqslant & \dis K e^{-c \tilde{\omega} \left|\lambda + \frac{k^2}{4}\right|^{1/4}} \\ \ecart

\dis\|L_\lambda^2 e^{cL_\lambda}\|_{\L(X)} &=&\dis \|-Q_\lambda e^{-c\sqrt{-Q_\lambda}}\|_{\L(X)} & \leqslant & \dis K e^{-c \tilde{\omega} \left|\lambda + \frac{k^2}{4}\right|^{1/4}},
\end{array}\right.$$
hence
$$\begin{array}{lll}
\dis \|e^{c(L_\lambda+M_\lambda)}\|_{\L(X)}  =  \|e^{-c(\sqrt{-Q_\lambda}+\sqrt{-P_\lambda})}\|_{\L(X)} & = & \dis \|e^{-c\sqrt{-Q_\lambda}} e^{-c\sqrt{-P_\lambda}}\|_{\L(X)} \\ \ecart
& \leqslant & \dis \|e^{-c\sqrt{-Q_\lambda}}\|_{\L(X)} \|e^{-c\sqrt{-P_\lambda}}\|_{\L(X)} \\ \ecart
& \leqslant & \dis K^2 e^{-2 c \tilde{\omega} \left|\lambda + \frac{k^2}{4}\right|^{1/4}},
\end{array}$$
where $L_\lambda$ and $M_\lambda$ are defined by \eqref{LM}. Moreover, from \refL{Lem LM^-1}, there exists $C > 0$, which is independent of $\lambda$, such that
$$\left\{\begin{array}{lllll}
\dis\left\|\left(L_\lambda+M_\lambda\right)^{2} M_\lambda^{-2} \right\|_{\L(X)} &  \leqslant & \dis \left\|L_\lambda M_\lambda^{-1}\right\|^2_{\L(X)} + 2\left\|L_\lambda M_\lambda^{-1}\right\| _{\L(X)} + 1 & \leqslant & C \\ \ecart
\dis\left\|\left(L_\lambda+M_\lambda\right)^{2} L_\lambda^{-2} \right\|_{\L(X)} & \leqslant & \dis \left\|M_\lambda L_\lambda^{-1}\right\|_{\L(X)}^2 + 2\left\|M_\lambda L_\lambda^{-1}\right\|_{\L(X)} + 1 & \leqslant & C.
\end{array}\right.$$
Furthermore, from \eqref{B lambda}, we have
\begin{equation}\label{estim B-1}
\left\|B_\lambda^{-1}\right\|_{\L(X)} \leqslant \frac{1}{2\sqrt{\left|\lambda + \frac{k^2}{4}\right|}}.
\end{equation}
Then, due to \eqref{estim B-1}, we obtain 
$$\begin{array}{lll}
M_{B_\lambda} & = & \max\left(\left\|T_{\lambda}^-\right\|_{\L(X)}, \left\|T_{\lambda}^+\right\|_{\L(X)} \right) \\ \ecart

&\leqslant &\dis ~~\,\left\Vert e^{c \left( L_\lambda +M_\lambda\right) }\right\Vert _{\L(X)} + \left\Vert B_\lambda^{-1}\left( L_\lambda +M_\lambda \right)^{2}M_\lambda^{-2}\right\Vert _{\mathcal{L}\left( X\right) }\left\Vert M_\lambda^{2}e^{c M_\lambda}\right\Vert _{\mathcal{L}\left( X\right) } \\ \ecart
&&+\left\Vert B_\lambda^{-1}\left( L_\lambda+M_\lambda\right) ^{2}L_\lambda^{-2}\right\Vert _{\mathcal{L}(X)}\left\Vert L_\lambda^{2}e^{c L_\lambda}\right\Vert _{\mathcal{L}\left( X\right) } \\ \\

&\leqslant &\dis ~~\,\left\Vert e^{c \left( L_\lambda+M_\lambda\right) }\right\Vert _{\L(X)} +  \left\|B_\lambda^{-1}\right\|_{\L(X)}\left\Vert\left( L_\lambda+M_\lambda\right)^{2}M_\lambda^{-2}\right\Vert _{\mathcal{L}\left( X\right) }\left\Vert M_\lambda^{2}e^{c M_\lambda}\right\Vert _{\mathcal{L}\left( X\right) } \\ \ecart
&& + \left\|B_\lambda^{-1}\right\|_{\L(X)}\left\Vert\left( L_\lambda+M_\lambda\right) ^{2}L_\lambda^{-2}\right\Vert _{\mathcal{L}(X)}\left\Vert L_\lambda^{2}e^{c L_\lambda}\right\Vert _{\mathcal{L}\left( X\right) } \\ \\

&\leqslant &\dis ~~\,\left\Vert e^{c \left( L_\lambda+M_\lambda\right) }\right\Vert _{\L(X)} +  \left\|B_\lambda^{-1}\right\|_{\L(X)}\left( 1 + \left\Vert L_\lambda
M_\lambda^{-1}\right\Vert _{\mathcal{L}\left( X\right) }\right)^2 \left\Vert M_\lambda^{2}e^{c M_\lambda}\right\Vert _{\mathcal{L}\left( X\right) } \\ \ecart
&& + \left\|B_\lambda^{-1}\right\|_{\L(X)}\left( 1 + \left\Vert M_\lambda
L_\lambda^{-1}\right\Vert _{\mathcal{L}(X)}\right)^2 \left\Vert L_\lambda^{2}e^{c L_\lambda}\right\Vert _{\mathcal{L}\left( X\right) } \\ \\

& \leqslant & \dis K^2 e^{-2c \tilde{\omega} \left|\lambda + \frac{k^2}{4}\right|^{1/4}} + 
\frac{CK}{\sqrt{\left|\lambda + \frac{k^2}{4}\right|}}\, e^{-c \tilde{\omega} \left|\lambda + \frac{k^2}{4}\right|^{1/4}}.
\end{array}$$
Then, there exists $\dis r > 0$ such that, for any $\dis -\frac{k^2}{4} + \left(\CC \setminus \left(\overline{B\left(0,r\right)} \cup \overline{S_{2\theta_A}}\right) \right)$, we have
\begin{equation}\label{minoration lambda + k2/4}
\left|\lambda + \frac{k^2}{4}\right| \geqslant r > 0,
\end{equation}
and
$$M_{B_\lambda} = \max\left(\left\|T_{\lambda}^-\right\|_{\L(X)}, \left\|T_{\lambda}^+\right\|_{\L(X)} \right) \leqslant \frac{1}{2} < 1.$$
For such $\lambda$, we deduce that $U_\lambda = I - T_{\lambda}^-$ and $V_\lambda = I - T_{\lambda}^+$ are invertible with bounded inverse, with
\begin{equation}\label{estim U-1 V-1}
\|U_\lambda^{-1}\|_{\L(X)} \leqslant 2 \quad \text{and} \quad \|V_\lambda^{-1}\|_{\L(X)} \leqslant 2,
\end{equation} 
which involves that $(H_6)$ of \refs{sect hyp gen} holds.

Finally, from \refT{thmprinc gen}, there exists a unique classical solution of \eqref{EDA}-(BCi)$_0$. Hence 
$$-\frac{k^2}{4} + \left(\CC \setminus \left(\overline{B\left(0,r\right)} \cup \overline{S_{2\theta_A}}\right) \right) \subset \rho(-\A_i),$$ 
which gives the result.

If in addition, we assume $(\H_5)$, then from \cite{LMMT}, Theorem 2.2, p. 355 and Theorem~2.5, p.~356-357, we obtain that $0 \in \rho(-\A_i)$, $i=3,4$.
\end{proof}

\subsection{Norm estimates}\label{sect estim norme}

In this section, we focus on the norm estimates of the classical solution of problem \eqref{EDA}-(BCi)$_0$, for $i = 1,2,3,4,5$. To this end, we adapted the following technical useful results from \cite{FLMT}. 
\begin{Lem}[\cite{FLMT}]\label{Lem LMT} 
Assume $(\H_1)$, $(\H_2)$ and $(\H_3)$. Let $f\in L^p(a,b;X)$ with $1<p<+\infty$. Then, for all $\varphi \in (0, \pi-\theta_A)$, we have
\begin{enumerate}
\item $-A$ is of type $\varphi$,

\item For all $\mu \in \overline{S_\varphi} \subset \rho(-A)$ and all $x \in [a,b]$, we set
$$I_{\mu,f}(x) =  \int_a^x e^{-(x-s)\sqrt{-A+\mu I}} f(s)~ ds \quad \text{and} \quad J_{\mu,f}(x) = \int_x^b e^{-(s-x)\sqrt{-A+\mu I}} f(s)~ ds.$$
Then, for all $\mu \in \overline{S_\varphi}$, we have
$$\left\|I_{\mu,f}\right\|_{L^p(a,b;X)} \leqslant \frac{C}{\sqrt{1+|\mu|}}\,\left\|f\right\|_{L^p(a,b;X)} \quad \text{and} \quad \left\|J_{\mu,f}\right\|_{L^p(a,b;X)} \leqslant \frac{C}{\sqrt{1+|\mu|}}\,\left\|f\right\|_{L^p(a,b;X)},$$
where $C$ is a positive constant independent of $f$ and $\mu$.
\end{enumerate}

\end{Lem}
Statement 1. is given by \refR{Rem type phi} and statement 2. is proved in \cite{FLMT}, Lemma 4.6. 

\begin{Lem}\label{Lem LMT 2}
Assume $(\H_1)$, $(\H_2)$ and $(\H_3)$. Let $\varphi \in (0, \pi-\theta_A)$ fixed. Then, there exists $C > 0$, such that, for all $\eta,\mu \in \overline{S_\varphi} \subset \rho(-A)$ and all $f\in L^p(a,b;X)$ with $1<p<+\infty$, we have
\begin{enumerate}
\item $\dis \left\|e^{-(.-a)\sqrt{-A + \eta I}} \int_a^b e^{-(s-a)\sqrt{-A+\mu I}} f(s)~ ds\right\|_{L^p(a,b;X)} \leqslant \left(\frac{C}{\sqrt{1+|\mu|}} + \frac{C}{\sqrt{1+|\eta|}} \right)\left\|f\right\|_{L^p(a,b;X)},$ 

\item $\dis \left\|e^{-(.-a)\sqrt{-A + \eta I}} \int_a^b e^{-(b-s)\sqrt{-A+\mu I}} f(s)~ ds \right\|_{L^p(a,b;X)} \leqslant \left(\frac{C}{\sqrt{1+|\mu|}} + \frac{C}{\sqrt{1+|\eta|}} \right) \left\|f\right\|_{L^p(a,b;X)},$

\item $\dis \left\|e^{-(b-.)\sqrt{-A + \eta I}}  \int_a^b e^{-(b-s)\sqrt{-A+\mu I}} f(s)~ ds\right\|_{L^p(a,b;X)} \leqslant \left(\frac{C}{\sqrt{1+|\mu|}} + \frac{C}{\sqrt{1+|\eta|}} \right) \left\|f\right\|_{L^p(a,b;X)},$

\item $\dis \left\|e^{-(b-.)\sqrt{-A + \eta I}} \int_a^b e^{-(s-a)\sqrt{-A+\mu I}} f(s)~ ds\right\|_{L^p(a,b;X)} \leqslant \left(\frac{C}{\sqrt{1+|\mu|}} + \frac{C}{\sqrt{1+|\eta|}} \right)\left\|f\right\|_{L^p(a,b;X)}.$
\end{enumerate}
\end{Lem}
\begin{proof}
We first focus on statement 1. For all $x \in [a,b]$, setting 
$$v(x) = e^{-(x-a)\sqrt{-A + \eta I}} \int_a^b e^{-(s-a)\sqrt{-A+\mu I}} f(s)~ ds,$$ 
we have
$$\begin{array}{lll}
\dis v(x) & = & \dis e^{-(x-a)\sqrt{-A + \eta I}} \int_a^x e^{-(s-a)\sqrt{-A+\mu I}} f(s)~ ds \\ \ecart
&& + \dis e^{-(x-a)\sqrt{-A + \eta I}} \int_x^b e^{-(s-a)\sqrt{-A+\mu I}} f(s)~ ds \\ \\

 & = & \dis  \int_a^x e^{-(x-s)\sqrt{-A + \eta I}} e^{-(s-a)\left(\sqrt{-A+\mu I} + \sqrt{-A + \eta I}\right)} f(s)~ ds \\ \ecart
&& + \dis e^{-(x-a)\sqrt{-A + \eta I}} e^{-(x-a)\sqrt{-A+\mu I}}\int_x^b e^{-(s-x)\sqrt{-A+\mu I}} f(s)~ ds \\ \\

& = & \dis I_{\eta,g} (x) + e^{-(x-a)\sqrt{-A + \eta I}} e^{-(x-a)\sqrt{-A+\mu I}} J_{\mu,f} (x),
\end{array}$$
where $g(s) = e^{-(s-a)\left(\sqrt{-A+\mu I} + \sqrt{-A + \eta I}\right)} f(s)$. Then from (38) in \cite{FLMT}, there exists $C > 0$, independent of $\eta$ and $\mu$, such that
$$\left\| e^{-(x-a)\sqrt{-A + \eta I}} e^{-(x-a)\sqrt{-A + \mu I}}\right\|_{\L(X)} \leqslant C \quad \text{and} \quad \left\|g\right\|_{L^p(a,b;X)} \leqslant C \left\|f\right\|_{L^p(a,b;X)}.$$
Finally \refL{Lem LMT} gives statement 1. The other statements are obtained, in the same way. 

Note that statements 1. and 2. have been proved in Lemma 5.6 in \cite{FLMT}, in the case when $a=0$, $b=1$ and $\eta = \mu$.
\end{proof}
\begin{Rem}\label{Rem LMT}
Note that, for $\dis\lambda \in -\frac{k^2}{4} + \left(\CC \setminus\overline{S_{2\theta_A}} \right)$, we have
$$M_\lambda = - \sqrt{-A_{k/2} + \mu I} \quad \text{and} \quad L_\lambda = - \sqrt{-A_{k/2} + \eta I},$$
with $\dis\mu = - \eta = -i\sqrt{-\lambda - \frac{k^2}{4}}$. But, from \refL{Lem sect}, we deduce that $\mu,\eta \in S_{\pi-\theta_A}$, thus we can apply \refL{Lem LMT 2}.
\end{Rem}
\begin{Lem}\label{Lem FLMT}
Let $t_0 > 0$ fixed. Then 
\begin{enumerate}
\item $\CC \setminus \overline{S_{\theta_A}} \subset \rho(A_{k/2})$ and for any $\nu \in (0, \pi-\theta_A)$, there exists $C_\nu > 0$ such that
$$\left\|\left(-A_{k/2} - \mu I\right)^{-1}\right\|_{\L(X)} \leqslant \frac{C_{\nu}}{|\mu|}, \quad \mu \in S_{\nu}.$$

\item There exists $C_{t_0} > 0$ such that, for any $\mu \in S_{\pi-\theta_A}$, we have 
$$\left\|\left(I \pm e^{t_0 H_\mu}\right)^{-1} \right\|_{\L(X)} \leqslant C_{t_0},$$
where $H_\mu := - \sqrt{-A_{k/2} - \mu I}$.
\end{enumerate} 
\end{Lem}

\begin{proof}\hfill
\begin{enumerate}
\item The result follows from Theorem 2, p. 437 in \cite{pruss-sohr} since $-A_{k/2} \in \text{BIP}\,(X,\theta_A)$.

\item We apply the same technique as in the proof of Lemma 5.1 in \cite{FLMT}: $A$ is replaced by $A_{k/2}$ and $e^{2A}$ is replaced by $e^{t_0 A_{k/2}}$.
\end{enumerate}
\end{proof}

Due to the previous study, we can apply the results of \refs{sect linear steady case} and then obtain that, for each $i = 1,2,3,4,5$, problem \eqref{EDA}-(BCi)$_0$ admits a unique classical solution $u_i$ given by \eqref{Repre u} in \refP{Proposition 1}, where $M$ and $L$ are replaced by $M_{\lambda}$ and $L_\lambda$. We need to give estimates on $u_i$ which take into account the dependence on $\lambda$, but $u_i$ contains $F_{0,f}$ and $F'_{0,f}$. To this end, we give the following result.

\begin{Lem}\label{Lem estim norme F'}
Let $\gamma \in \{a,b\}$. Assume that $(\H_1)$, $(\H_2)$, $(\H_3)$ and $(\H_4)$ hold. Then, there exists $C > 0$, such that, for all $\dis\lambda \in -\frac{k^2}{4} + \left(\CC \setminus \overline{S_{2\theta_A}}\right)$ and all $f \in L^p(a,b;X)$, we have
$$\left\{\begin{array}{lll}
\dis\left\|v_0\right\|_{L^p(a,b;X)} & \leqslant & \dis \frac{C }{1+\sqrt{\left|\lambda + \frac{k^2}{4}\right|}}\,\left\|f\right\|_{L^p(a,b;X)} \\ \\

\left\|L_\lambda v_0(.)\right\|_{L^p(a,b;X)} & \leqslant & \dis\frac{C}{\sqrt{1+\sqrt{\left|\lambda + \frac{k^2}{4}\right|}}} \,\left\|f\right\|_{L^p(a,b;X)}.
\end{array}\right.$$
Moreover, for $T_\lambda = M_\lambda$ or $L_\lambda$, we have
$$\left\{\begin{array}{lll}
\dis\left\|e^{(.-a)T_\lambda} F'_{0,f}(\gamma)\right\|_{L^p(a,b;X)} &\leqslant&\dis \frac{C}{\left(1 + \sqrt{\left|\lambda + \frac{k^2}{4}\right|}\right)^{\frac{3}{2}}}\,\|f\|_{L^p(a,b;X)} \\ \\

\dis\left\|e^{(b-.)T_\lambda} F'_{0,f}(\gamma)\right\|_{L^p(a,b;X)} &\leqslant& \dis \frac{C}{\left(1 + \sqrt{\left|\lambda + \frac{k^2}{4}\right|}\right)^{\frac{3}{2}}}\,\|f\|_{L^p(a,b;X)},
\end{array}\right.$$
and
$$\left\{\begin{array}{lll}
\dis\left\|e^{(.-a)T_\lambda} L_\lambda F'_{0,f}(\gamma)\right\|_{L^p(a,b;X)} &\leqslant&\dis \frac{C}{1 + \sqrt{\left|\lambda + \frac{k^2}{4}\right|}}\,\|f\|_{L^p(a,b;X)} \\ \\

\dis\left\|e^{(b-.)T_\lambda} L_\lambda F'_{0,f}(\gamma)\right\|_{L^p(a,b;X)} &\leqslant& \dis \frac{C}{1 + \sqrt{\left|\lambda + \frac{k^2}{4}\right|}}\,\|f\|_{L^p(a,b;X)},
\end{array}\right.$$
where for all $x \in [a,b]$, $F'_{0,f}$ and $v_0$ are given by
\begin{equation}\label{F'0}
\begin{array}{rll}
F'_{0,f}(x) & = &  - \dis \frac{1}{2} \left(e^{(x-a)M_\lambda} + e^{(b-x)M_\lambda} e^{cM_\lambda} \right) Z \int_a^b e^{(s-a)M_\lambda} v_0(s)~ ds \\ \ecart
&& + \dis \frac{1}{2} \left(e^{(b-x)M_\lambda} + e^{(x-a)M_\lambda} e^{cM_\lambda}\right) Z \int_a^b e^{(b-s)M_\lambda} v_0(s)~ ds \\ \\
&& \dis + \frac{1}{2} \int_a^x e^{(x-s)M_\lambda} v_0(s)~ ds - \frac{1}{2} \int_x^b e^{(s-x)M_\lambda} v_0(s)~ ds, 
\end{array}
\end{equation}
and
\begin{equation}\label{v0}
\begin{array}{rll}
v_0(x) & = & \dis\frac{1}{2} \left(e^{(b-x)L_\lambda}  e^{cL_\lambda} - e^{(x-a)L_\lambda}\right) W L_\lambda^{-1} \dis\int_a^b e^{(s-a)L_\lambda} f(s)~ ds \\ \ecart
&& + \dis \frac{1}{2} \left(e^{(x-a)L_\lambda} e^{cL_\lambda} - e^{(b-x)L_\lambda} \right) W L_\lambda^{-1} \int_a^b e^{(b-s)L_\lambda} f(s)~ ds \\ \ecart
&& +\dis \frac{1}{2} L_\lambda^{-1} \dis\int_a^x e^{(x-s)L_\lambda} f(s)~ ds + \frac{1}{2} L_\lambda^{-1} \int_x^b e^{(s-x)L_\lambda} f(s)~ ds,
\end{array}
\end{equation}
with $Z = \left(I-e^{2cM_\lambda} \right)^{-1}$ and $W = \left(I-e^{2cL_\lambda}\right)^{-1}$.
\end{Lem}
\begin{proof}
In the sequel, $C > 0$ will denote various constants which are independent of $f$ and $\lambda$. 

Let $t_0>0$ fixed. For all $\dis\lambda \in -\frac{k^2}{4} + \left(\CC \setminus \overline{S_{2\theta_A}}\right)$, we have
$$\left\{\begin{array}{llllr}
M_\lambda &=& - \sqrt{-A_{k/2} - i \sqrt{-\lambda - \frac{k^2}{4}} I}, & \text{with} &i\sqrt{-\lambda - \frac{k^2}{4}} \in S_{\pi-\theta_A}\, \\ \ecart
L_\lambda &=& - \sqrt{-A_{k/2} - \left(- i \sqrt{-\lambda - \frac{k^2}{4}} I\right)}, & \text{with} &-i\sqrt{-\lambda - \frac{k^2}{4}} \in S_{\pi-\theta_A}.
\end{array}\right.$$
Then, from \refL{Lem FLMT}, there exists $C_{t_0} > 0$ such that for all $\lambda \in -\frac{k^2}{4} + \left(\CC \setminus \overline{S_{2\theta_A}}\right)$, we obtain
\begin{equation}\label{estim Z et W}
\left\|\left(I \pm e^{t_0 M_\lambda}\right)^{-1}\right\|_{\L(X)} \leqslant C_{t_0} \quad \text{and} \quad \left\|\left(I \pm e^{t_0 L_\lambda}\right)^{-1}\right\|_{\L(X)} \leqslant C_{t_0}.
\end{equation}
Moreover, since $-M_\lambda \in$ BIP$(X,\theta_M)$ and $-L_\lambda \in$ BIP$(X,\theta_L)$, with $\theta_M, \theta_L \in [0,\pi/2)$, from \refL{Lem DY}, we deduce that
\begin{equation}\label{estim pazy}
\left\|e^{c M_\lambda}\right\|_{\L(X)} \leqslant C \quad \text{and} \quad \left\|e^{c L_\lambda}\right\|_{\L(X)} \leqslant C.
\end{equation}
Furthermore, from \cite{dore-yakubov}, Lemma 2.6, statement a), p. 104, we have
\begin{equation}\label{estim M-1}
\left\|M_\lambda^{-1}\right\|_{\L(X)} \leqslant \frac{C}{\sqrt{1+\left|- i\sqrt{-\lambda - \frac{k^2}{4}}\right|}} = \frac{C}{\sqrt{1+\sqrt{\left|\lambda + \frac{k^2}{4}\right|}}},
\end{equation}
and
\begin{equation}\label{estim L-1}
\left\|L_\lambda^{-1}\right\|_{\L(X)} \leqslant \frac{C}{\sqrt{1+\left|i \sqrt{-\lambda - \frac{k^2}{4}}\right|}} = \frac{C}{\sqrt{1+\sqrt{\left|\lambda + \frac{k^2}{4}\right|}}}.
\end{equation}
From \eqref{v0}, since $v_0(x) \in D(L_\lambda)$, for $x \in [a,b]$, we have
$$\begin{array}{lll}
L_\lambda v_0(x) &=& \dis \frac{1}{2} W e^{cL_\lambda} e^{(b-x)L_\lambda} \int_a^b e^{(s-a)L_\lambda} f(s)~ ds  - \frac{1}{2} W e^{(x-a)L_\lambda} \int_a^b e^{(s-a)L_\lambda} f(s)~ ds  \\ \ecart

&&+ \dis \frac{1}{2}  W e^{cL_\lambda} e^{(x-a)L_\lambda} \int_a^b e^{(b-s)L_\lambda} f(s)~ ds - \frac{1}{2} W e^{(b-x)L_\lambda} \int_a^b e^{(b-s)L_\lambda} f(s)~ ds \\ \ecart

&&+ \dis \frac{1}{2} \int_a^x e^{(x-s)L_\lambda} f(s)~ ds + \frac{1}{2} \int_x^b e^{(s-x)L_\lambda} f(s)~ ds.
\end{array}$$
Hence
$$\begin{array}{rll}
\left\|L_\lambda v_0(.)\right\|_{L^p(a,b;X)} & \leqslant & \dis \frac{1}{2} \left\|W\right\|_{\L(X)} \left\|e^{cL_\lambda}\right\|_{\L(X)} \left\|e^{(b-.)L_\lambda} \int_a^b e^{(s-a)L_\lambda} f(s)~ ds\right\|_{L^p(a,b;X)} \\ \ecart
&& \dis + \frac{1}{2} \left\|W\right\|_{\L(X)}  \left\|e^{(.-a)L_\lambda} \int_a^b e^{(s-a)L_\lambda} f(s)~ ds\right\|_{L^p(a,b;X)} \\ \ecart

&& + \dis \frac{1}{2} \left\|W\right\|_{\L(X)}  \left\|e^{cL_\lambda}\right\|_{\L(X)} \left\|e^{(.-a)L_\lambda} \int_a^b e^{(b-s)L_\lambda} f(s)~ ds\right\|_{L^p(a,b;X)} \\ \ecart

&& \dis + \frac{1}{2} \left\|W\right\|_{\L(X)} \left\|e^{(b-.)L_\lambda} \int_a^b e^{(b-s)L_\lambda} f(s)~ ds\right\|_{L^p(a,b;X)} \\ \ecart

&& + \dis \frac{1}{2} \left\|\int_a^. e^{(.-s)L_\lambda} f(s)~ ds\right\|_{L^p(a,b;X)} + \frac{1}{2} \left\|\int_.^b e^{(s-.)L_\lambda} f(s)~ ds\right\|_{L^p(a,b;X)},
\end{array}$$
and from \eqref{estim Z et W}, \eqref{estim pazy}, \refL{Lem LMT}, \refL{Lem LMT 2} and \refR{Rem LMT}, we have
\begin{equation}\label{estim L v0 proof}
\left\|L_\lambda v_0(.)\right\|_{L^p(a,b;X)} \leqslant \frac{C}{\sqrt{1+\sqrt{\left|\lambda + \frac{k^2}{4}\right|}}} \,\left\|f\right\|_{L^p(a,b;X)}.
\end{equation}
Moreover, since $v_0 = L_\lambda^{-1} L_\lambda v_0(.)$, it follows that 
$$\left\| v_0\right\|_{L^p(a,b;X)} \leqslant \left\|L_\lambda^{-1}\right\|_{\L(X)} \left\|L_\lambda v_0(.)\right\|_{L^p(a,b;X)}.$$
Thus, from \eqref{estim L-1} and \eqref{estim L v0 proof}, it follows that
\begin{equation}\label{estim v0 proof}
\left\|v_0\right\|_{L^p(a,b;X)} \leqslant \frac{C}{1+\sqrt{\left|\lambda + \frac{k^2}{4}\right|}} \,\left\|f\right\|_{L^p(a,b;X)}.
\end{equation}
In the same way, from \eqref{F'0}, we have
\begin{equation*}
\begin{array}{l}
L_\lambda F'_{0,f}(\gamma) = \\ \ecart
\dis - \frac{1}{2} Z e^{(\gamma-a)M_\lambda} L_\lambda\int_a^b e^{(s-a)M_\lambda} v_0(s)~ ds - \frac{1}{2} Z e^{cM_\lambda} e^{(b-\gamma)M_\lambda} L_\lambda \int_a^b e^{(s-a)M_\lambda} v_0(s)~ ds \\ \ecart

 + \dis \frac{1}{2} Z e^{(b-\gamma)M_\lambda} L_\lambda \int_a^b e^{(b-s)M_\lambda} v_0(s)~ ds + \frac{1}{2} Z e^{cM_\lambda} e^{(\gamma-a)M_\lambda}L_\lambda\int_a^b e^{(b-s)M_\lambda} v_0(s)~ ds \\ \ecart

 \dis + \frac{1}{2} L_\lambda\int_a^\gamma e^{(\gamma-s)M_\lambda} v_0(s)~ ds - \frac{1}{2} L_\lambda \int_\gamma^b e^{(s-\gamma)M_\lambda} v_0(s)~ ds,
\end{array}
\end{equation*}
hence
\begin{equation*}
\begin{array}{l}
L_\lambda F'_{0,f}(\gamma) = \\ \ecart
\dis - \frac{1}{2} Z e^{(\gamma-a)M_\lambda}\int_a^b e^{(s-a)M_\lambda} L_\lambda v_0(s)~ ds - \frac{1}{2} Z e^{cM_\lambda} e^{(b-\gamma)M_\lambda} \int_a^b e^{(s-a)M_\lambda} L_\lambda v_0(s)~ ds \\ \ecart

+ \dis \frac{1}{2} Z e^{(b-\gamma)M_\lambda} \int_a^b e^{(b-s)M_\lambda} L_\lambda v_0(s)~ ds + \frac{1}{2} Z e^{cM_\lambda} e^{(\gamma-a)M_\lambda}\int_a^b e^{(b-s)M_\lambda} L_\lambda v_0(s)~ ds \\ \ecart

 \dis + \frac{1}{2} \int_a^\gamma e^{(\gamma-s)M_\lambda} L_\lambda v_0(s)~ ds - \frac{1}{2}  \int_\gamma^b e^{(s-\gamma)M_\lambda} L_\lambda v_0(s)~ ds.
\end{array}
\end{equation*}
Moreover, we have
\begin{equation*}
\begin{array}{l}
\left\|e^{(.-a)T_\lambda} L_\lambda F'_{0,f}(\gamma)\right\|_{L^p(a,b;X)}  \leqslant \\ \ecart 

\dis \frac{1}{2} \left\|Z\right\|_{\L(X)} \left\|e^{(\gamma-a)M_\lambda}\right\|_{\L(X)} \left\|e^{(.-a)T_\lambda} \int_a^b e^{(s-a)M_\lambda} L_\lambda v_0(s)~ ds\right\|_{L^p(a,b;X)} \\ \ecart

 \dis + \frac{1}{2} \left\|Z\right\|_{\L(X)} \left\|e^{cM_\lambda}\right\|_{\L(X)} \left\|e^{(b-\gamma)M_\lambda}\right\|_{\L(X)} \left\|e^{(.-a)T_\lambda} \int_a^b e^{(s-a)M_\lambda} L_\lambda v_0(s)~ ds\right\|_{L^p(a,b;X)} \\ \ecart

 + \dis \frac{1}{2} \left\|Z\right\|_{\L(X)} \left\|e^{(b-\gamma)M_\lambda} \right\|_{\L(X)} \left\|e^{(.-a)T_\lambda} \int_a^b e^{(b-s)M_\lambda} L_\lambda v_0(s)~ ds\right\|_{L^p(a,b;X)} \\ \ecart

 + \dis \frac{1}{2} \left\|Z\right\|_{\L(X)} \left\|e^{cM_\lambda}\right\|_{\L(X)} \left\|e^{(\gamma-a)M_\lambda}\right\|_{\L(X)} \left\|e^{(.-a)T_\lambda} \int_a^b e^{(b-s)M_\lambda} L_\lambda v_0(s)~ ds\right\|_{L^p(a,b;X)} \\ \ecart

 \dis + \frac{1}{2} \left\|e^{(.-a)T_\lambda}\int_a^\gamma e^{(\gamma-s)M_\lambda} L_\lambda v_0(s)~ ds\right\|_{L^p(a,b;X)} + \frac{1}{2} \left\|e^{(.-a)T_\lambda}\int_\gamma^b e^{(s-\gamma)M_\lambda} L_\lambda v_0(s)~ ds\right\|_{L^p(a,b;X)}.
\end{array}
\end{equation*}
Since $L_\lambda v_0(.) \in L^p(a,b;X)$, from \eqref{estim Z et W}, \eqref{estim pazy}, \refL{Lem LMT 2} and \refR{Rem LMT}, we have
\begin{equation*}
\left\|e^{(.-a)T_\lambda} L_\lambda F'_{0,f}(\gamma)\right\|_{L^p(a,b;X)} \leqslant \frac{C}{\sqrt{1+\sqrt{\left|\lambda + \frac{k^2}{4}\right|}}} \left\|L_\lambda v_0(.)\right\|_{L^p(a,b;X)}.
\end{equation*}
Moreover, from \eqref{estim L-1}, we have 
$$\begin{array}{lll}
\dis\left\|e^{(.-a)T_\lambda} F'_{0,f}(\gamma)\right\|_{L^p(a,b;X)} &=& \dis  \left\|e^{(.-a)T_\lambda} L_\lambda L_\lambda^{-1} F'_{0,f}(\gamma)\right\|_{L^p(a,b;X)} \\ \ecart

& \leqslant & \dis \left\|L_\lambda^{-1}\right\|_{\L(X)}\left\|e^{(.-a)T_\lambda} L_\lambda F'_{0,f}(\gamma)\right\|_{L^p(a,b;X)} \\ \ecart

& \leqslant & \dis\frac{C}{1+\sqrt{\left|\lambda + \frac{k^2}{4}\right|}} \left\|L_\lambda v_0(.)\right\|_{L^p(a,b;X)}.
\end{array}$$
In the same way, we obtain that
\begin{equation*}
\left\|e^{(b-.)T_\lambda} L_\lambda F'_{0,f}(\gamma)\right\|_{L^p(a,b;X)} \leqslant \frac{C}{\sqrt{1+\sqrt{\left|\lambda + \frac{k^2}{4}\right|}}} \left\|L_\lambda v_0(.)\right\|_{L^p(a,b;X)},
\end{equation*}
and
\begin{equation*}
\left\|e^{(b-.)T_\lambda} F'_{0,f}(\gamma)\right\|_{L^p(a,b;X)} \leqslant \frac{C}{1+\sqrt{\left|\lambda + \frac{k^2}{4}\right|}} \left\|L_\lambda v_0(.)\right\|_{L^p(a,b;X)}.
\end{equation*}
Finally, from \eqref{estim L v0 proof}, we obtain the expected results.
\end{proof}

\begin{Prop}\label{Prop estim norme}
Assume that $(\H_1)$, $(\H_2)$, $(\H_3)$ and $(\H_4)$ hold. Then
\begin{enumerate}
\item for $i=1,2,5$, there exists constants $C_i > 0$, such that, for all $\dis\lambda \in -\frac{k^2}{4} + \left(\CC \setminus \overline{S_{2\theta_A}}\right)$ and all $f\in L^p(a,b;X)$, we have
$$\left\|(-\A_i - \lambda I)^{-1}f \right\|_{\L(X)} \leqslant \frac{C_i}{1 + \left|\lambda + \frac{k^2}{4}\right|}\, \|f\|_{L^p(a,b;X)},$$

\item for $i=3,4$, there exists $C_i > 0$, such that, for all $\dis\lambda \in -\frac{k^2}{4} + \left(\CC \setminus \left(\overline{B\left(0,r\right)} \cup \overline{S_{2\theta_A}}\right)\right)$ and all $f \in L^p(a,b;X)$, we have
$$\left\|(-\A_i - \lambda I)^{-1}f\right\|_{\L(X)} \leqslant \frac{C_i}{1 + \left|\lambda + \frac{k^2}{4}\right|}\,\|f\|_{L^p(a,b;X)}.$$
\end{enumerate}
\end{Prop}

\begin{proof}

In the sequel, $C > 0$ will denote various constants which are independent of $f$ and $\lambda$. Moreover, we will use that, for $i = 1,2,3,4,5$, we have $(-\A_i - \lambda I)^{-1}f = u_i$, where $u_i$ is the unique classical solution of \eqref{EDA}-(BCi)$_0$.
\begin{enumerate}
\item Let $\dis\lambda \in -\frac{k^2}{4} + \left(\CC \setminus \overline{S_{2\theta_A}}\right)$. 

We first focus on $-\A_1$. Our aim is to obtain that
$$\left\|(-\A_1 - \lambda I)^{-1}f \right\|_{\L(X)} = \| u_1 \|_{L^p(a,b;X)} \leqslant \frac{C}{1 + \left|\lambda + \frac{k^2}{4}\right|} \,\|f\|_{L^p(a,b;X)}.$$
To this end, we first recall, from \refT{Th Carre}, the expression of $u_1 = F_{0,f}$, given, for all $x \in [a,b]$, by 
\begin{equation}\label{F0}
\begin{array}{rll}
F_{0,f}(x) & = & \dis \frac{1}{2} \left(e^{(b-x)M_\lambda} e^{cM_\lambda} - e^{(x-a)M_\lambda}\right) Z M_\lambda^{-1} \int_a^b e^{(s-a)M_\lambda} v_0(s)~ ds \\ \ecart
&&+ \dis \frac{1}{2} \left(e^{(x-a)M_\lambda} e^{cM_\lambda} - e^{(b-x)M_\lambda}\right) Z M_\lambda^{-1} \int_a^b e^{(b-s)M_\lambda} v_0(s)~ ds \\ \ecart
&& + \dis \frac{1}{2} M_\lambda^{-1} \dis\int_a^x e^{(x-s)M_\lambda} v_0(s)~ ds + \frac{1}{2} M_\lambda^{-1} \int_x^b e^{(s-x)M_\lambda} v_0(s)~ ds,
\end{array}
\end{equation}
where $v_0$ is given by \eqref{v0} and $Z = \left(I-e^{2cM_\lambda} \right)^{-1}$.

Moreover, following the same step as in the proof of \refL{Lem estim norme F'}, we have
\begin{equation*}
\|F_{0,f}\|_{L^p(a,b;X)} \leqslant \frac{C}{1+ \sqrt{\left|\lambda + \frac{k^2}{4}\right|}}\, \|v_0\|_{L^p(a,b;X)},
\end{equation*}
and from \refL{Lem estim norme F'}, we have
\begin{equation*}
\left\|v_0\right\|_{L^p(a,b;X)} \leqslant \frac{C}{1 + \sqrt{\left|\lambda + \frac{k^2}{4}\right|}} \,\|f\|_{L^p(a,b;X)}.
\end{equation*}
Then, we obtain
\begin{equation}\label{estim u1}
\begin{array}{lll}
\dis \left\|u_1\right\|_{L^p(a,b;X)} = \left\|F_{0,f} \right\|_{L^p(a,b;X)} & \leqslant & \dis \frac{C}{1 + 2\,\sqrt{\left|\lambda + \frac{k^2}{4}\right|} + \left|\lambda + \frac{k^2}{4}\right|}\, \|f\|_{L^p(a,b;X)} \\ \ecart
& \leqslant & \dis \frac{C}{1 + \left|\lambda + \frac{k^2}{4}\right|}\, \|f\|_{L^p(a,b;X)}.
\end{array}
\end{equation}
From \refT{Th Carre}, since $u_5 = u_1$, the result follows for $u_5$. 

Our aim is now to show that 
$$\left\|(-\A_2 - \lambda I)^{-1}f \right\|_{\L(X)} = \| u_2 \|_{L^p(a,b;X)} \leqslant \frac{C}{1 + \left|\lambda + \frac{k^2}{4}\right|}\,\|f\|_{L^p(a,b;X)}.$$
We first recall, from \eqref{Repre u bis}, \eqref{ALPHA 1 2 3 4} and \eqref{PHI 1 2 tilde}, that $u_2$ is given, for all $x \in [a,b]$, by
\begin{equation*}
\begin{array}{rcl}
u_2(x)&=&\dis~~ \,\left(e^{(x-a) M_\lambda} - e^{(b-x) M_\lambda}\right)\alpha_1 + \left(e^{(x-a) L_\lambda} - e^{(b-x) L_\lambda}\right)\alpha_2 \\ \ecart
&&\dis+\left(e^{(x-a) M_\lambda} + e^{(b-x) M_\lambda}\right)\alpha_3 + \left(e^{(x-a) L_\lambda} + e^{(b-x) L_\lambda}\right)\alpha_4 + F_{0 ,f}(x),
\end{array}
\end{equation*}
where $F_{0,f} = u_1$ is given by \eqref{F0} and
\begin{equation*}
\left\{\begin{array}{rcl}
\alpha_1 & = & \dis - \frac{1}{2} \left(I + e^{cM_\lambda}\right)^{-1} M_\lambda^{-1} \left(F_{0,f}'(a) + F'_{0,f}(b)\right) \\ \ecart 

\alpha_2 & = & \dis 0 \\ \ecart

\alpha_3 & = & \dis - \frac{1}{2} \left(I - e^{cM_\lambda}\right)^{-1} M_\lambda^{-1} \left(F_{0,f}'(a) - F'_{0,f}(b)\right) \\ \ecart 

\alpha_4 & = & \dis 0.
\end{array}\right.
\end{equation*}
Thus, we have
\begin{equation*}
\begin{array}{rll}
\left\|u_2\right\|_{L^p(a,b;X)} & \leqslant & \dis ~~\frac{1}{2} \left\|\left(I + e^{cM_\lambda}\right)^{-1}\right\|_{\L(X)}\|M_\lambda^{-1}\|_{\L(X)} \left\|e^{(.-a)M_\lambda}F'_{0,f}(a)\right\|_{L^p(a,b;X)} \\ \ecart

&&\dis + \frac{1}{2} \left\|\left(I + e^{cM_\lambda}\right)^{-1}\right\|_{\L(X)}\|M_\lambda^{-1}\|_{\L(X)} \left\|e^{(.-a)M_\lambda}F'_{0,f}(b)\right\|_{L^p(a,b;X)} \\ \ecart

&& \dis +\frac{1}{2} \left\|\left(I + e^{cM_\lambda}\right)^{-1}\right\|_{\L(X)}\|M_\lambda^{-1}\|_{\L(X)} \left\|e^{(b-.)M_\lambda}F'_{0,f}(a)\right\|_{L^p(a,b;X)} \\ \ecart

&&\dis + \frac{1}{2} \left\|\left(I + e^{cM_\lambda}\right)^{-1}\right\|_{\L(X)}\|M_\lambda^{-1}\|_{\L(X)} \left\|e^{(b-.)M_\lambda}F'_{0,f}(b)\right\|_{L^p(a,b;X)} \\ \ecart

&& \dis + \frac{1}{2} \left\|\left(I - e^{cM_\lambda}\right)^{-1}\right\|_{\L(X)}\|M_\lambda^{-1}\|_{\L(X)} \left\|e^{(.-a)M_\lambda}F'_{0,f}(a)\right\|_{L^p(a,b;X)} \\ \ecart

&&\dis + \frac{1}{2} \left\|\left(I - e^{cM_\lambda}\right)^{-1}\right\|_{\L(X)}\|M_\lambda^{-1}\|_{\L(X)} \left\|e^{(.-a)M_\lambda}F'_{0,f}(b)\right\|_{L^p(a,b;X)} \\ \ecart

&& \dis +\frac{1}{2} \left\|\left(I - e^{cM_\lambda}\right)^{-1}\right\|_{\L(X)}\|M_\lambda^{-1}\|_{\L(X)} \left\|e^{(b-.)M_\lambda}F'_{0,f}(a)\right\|_{L^p(a,b;X)} \\ \ecart

&&\dis + \frac{1}{2} \left\|\left(I - e^{cM_\lambda}\right)^{-1}\right\|_{\L(X)}\|M_\lambda^{-1}\|_{\L(X)} \left\|e^{(b-.)M_\lambda}F'_{0,f}(b)\right\|_{L^p(a,b;X)} \\ \ecart
&& \dis + \left\|F_{0,f}\right\|_{L^p(a,b;X)},
\end{array}
\end{equation*}
and from \eqref{estim Z et W}, \eqref{estim M-1}, \refL{Lem estim norme F'} and \eqref{estim u1}, we have
$$\| u_2 \|_{L^p(a,b;X)} \leqslant \frac{C}{1 + 2\sqrt{\left|\lambda + \frac{k^2}{4}\right|} + \left|\lambda + \frac{k^2}{4}\right|}\,\|f\|_{L^p(a,b;X)} \leqslant \frac{C}{1 + \left|\lambda + \frac{k^2}{4}\right|}\,\|f\|_{L^p(a,b;X)}.$$

\item Let $\dis \lambda \in -\frac{k^2}{4} + \left(\CC \setminus \left(\overline{B\left(0,r\right)} \cup \overline{S_{2\theta_A}}\right)\right)$. 

Our aim is to obtain that
$$ \left\|(-\A_i - \lambda I)^{-1}f\right\|_{\L(X)} = \| u_i \|_{L^p(a,b;X)} \leqslant \frac{C}{1 + \left|\lambda + \frac{k^2}{4}\right|} \, \|f\|_{L^p(a,b;X)}, \quad \text{for } i=3,4.$$
Recall, from \eqref{Repre u bis}, \eqref{tildephi gen} and \eqref{ali gen}, that $u_3$ is given by
\begin{equation}\label{u3}
\begin{array}{rcl}
u_3(x)&=&\dis~~ \,\left(e^{(x-a) M_\lambda} - e^{(b-x) M_\lambda}\right)\alpha_1 + \left(e^{(x-a) L_\lambda} - e^{(b-x) L_\lambda}\right)\alpha_2 \\ \ecart
&&\dis+\left(e^{(x-a) M_\lambda} + e^{(b-x) M_\lambda}\right)\alpha_3 + \left(e^{(x-a) L_\lambda} + e^{(b-x) L_\lambda}\right)\alpha_4 + F_{0 ,f}(x),
\end{array}
\end{equation}
where $F_{0,f} = u_1$ is given by \eqref{F0} and
\begin{equation*}
\left\{\begin{array}{rcl}
\alpha_1 & = & \dis~~ \frac{1}{2} B_\lambda^{-1} (L_\lambda+M_\lambda) U_\lambda^{-1} \left(I - e^{cL_\lambda}\right) \left(F_{0,f}'(a) + F'_{0,f}(b)\right) \\ \ecart

\alpha_2 & = & \dis -\frac{1}{2} B_\lambda^{-1} (L_\lambda+M_\lambda) U_\lambda^{-1} \left(I - e^{cM_\lambda}\right) \left(F_{0,f}'(a) + F'_{0,f}(b)\right) \\ \ecart

\alpha_3 & = & \dis ~~\frac{1}{2} B_\lambda^{-1} (L_\lambda+M_\lambda) V_\lambda^{-1} \left(I + e^{cL_\lambda}\right) \left(F_{0,f}'(a) - F'_{0,f}(b)\right)  \\ \ecart

\alpha_4 & = & \dis -\frac{1}{2} B_\lambda^{-1} (L_\lambda+M_\lambda) V_\lambda^{-1} \left(I + e^{cM_\lambda}\right) \left(F_{0,f}'(a) - F'_{0,f}(b)\right).
\end{array}\right.
\end{equation*}
Moreover, let 
$$T_\lambda = M_\lambda ~\, \text{or} ~\, L_\lambda,$$
then from \eqref{estim B-1}, \eqref{estim U-1 V-1}, \eqref{estim Z et W}, \refL{Lem LM^-1} and \refL{Lem estim norme F'}, for $i=1,2,3,4$, we have
$$\begin{array}{lll}
\|e^{(.-a)T_\lambda} \alpha_i \|_{L(a,b;X)} & \leqslant & \dis \frac{C}{\sqrt{\left|\lambda + \frac{k^2}{4}\right|}} \left\|e^{(.-a)T_\lambda}(L_\lambda+M_\lambda)F'_{0,f}(a)\right\|_{L^p(a,b;X)} \\ \ecart

&& \dis + \frac{C}{\sqrt{\left|\lambda + \frac{k^2}{4}\right|}} \left\|e^{(.-a)T_\lambda}(L_\lambda+M_\lambda)F'_{0,f}(b)\right\|_{L^p(a,b;X)} \\ \\

& \leqslant & \dis \frac{C}{\sqrt{\left|\lambda + \frac{k^2}{4}\right|}} \left\|(L_\lambda+M_\lambda)L_\lambda^{-1}\right\|_{\L(X)} \left\|e^{(.-a)T_\lambda}L _\lambda F'_{0,f}(a)\right\|_{L^p(a,b;X)} \\ \ecart
&& \dis + \frac{C}{\sqrt{|\lambda + \frac{k^2}{4}|}} \left\|(L_\lambda+M_\lambda)L_\lambda^{-1}\right\|_{\L(X)} \left\|e^{(.-a)T_\lambda}L_\lambda F'_{0,f}(b)\right\|_{L^p(a,b;X)} \\ \\

& \leqslant & \dis \frac{C}{\sqrt{\left|\lambda + \frac{k^2}{4}\right|}} \left\|I + M_\lambda L_\lambda^{-1}\right\|_{\L(X)} \left\|e^{(.-a)T_\lambda}L _\lambda F'_{0,f}(a)\right\|_{L^p(a,b;X)} \\ \ecart

&& \dis +\frac{C}{\sqrt{\left|\lambda + \frac{k^2}{4}\right|}} \left\|I + M_\lambda L_\lambda^{-1}\right\|_{\L(X)} \left\|e^{(.-a)T_\lambda}L_\lambda F'_{0,f}(b)\right\|_{L^p(a,b;X)},
\end{array}$$
hence
$$\begin{array}{lll}
\|e^{(.-a)T_\lambda} \alpha_i \|_{L(a,b;X)} &\leqslant & \dis \frac{C}{\sqrt{\left|\lambda + \frac{k^2}{4}\right|}} \left( 1 + \left\|M_\lambda L_\lambda^{-1}\right\|_{\L(X)}\right) \left\|e^{(.-a)T_\lambda}L _\lambda F'_{0,f}(a)\right\|_{L^p(a,b;X)} \\ \ecart

&& \dis +  \frac{C}{\sqrt{\left|\lambda + \frac{k^2}{4}\right|}} \left( 1 + \left\|M_\lambda L_\lambda^{-1}\right\|_{\L(X)}\right) \left\|e^{(.-a)T_\lambda}L_\lambda F'_{0,f}(b)\right\|_{L^p(a,b;X)} \\ \\ 

& \leqslant & \dis \frac{C}{\sqrt{\left|\lambda + \frac{k^2}{4}\right|}} \left\|e^{(.-a)T_\lambda}L _\lambda F'_{0,f}(a)\right\|_{L^p(a,b;X)} \\ \ecart
&&\dis + \frac{C}{\sqrt{\left|\lambda + \frac{k^2}{4}\right|}} \left\|e^{(.-a)T_\lambda}L_\lambda F'_{0,f}(b)\right\|_{L^p(a,b;X)} \\ \\

& \leqslant & \dis \frac{C}{\sqrt{\left|\lambda + \frac{k^2}{4}\right|} + \left|\lambda + \frac{k^2}{4}\right|}\, \|f\|_{L^p(a,b;X)}.
\end{array}$$
In the same way, we obtain that
$$\left\|e^{(b-.)T_\lambda} \alpha_i \right\|_{L(a,b;X)} \leqslant \frac{C}{\sqrt{\left|\lambda + \frac{k^2}{4}\right|} + \left|\lambda + \frac{k^2}{4}\right|}\, \|f\|_{L^p(a,b;X)}.$$
Moreover, due to \eqref{minoration lambda + k2/4}, we have 
$$\sqrt{\left|\lambda+\frac{k^2}{4}\right|} \geqslant \sqrt{r} >0.$$
Hence
$$\frac{C}{\sqrt{\left|\lambda + \frac{k^2}{4}\right|} + \left|\lambda + \frac{k^2}{4}\right|} \leqslant \frac{C}{\sqrt{r} + \left|\lambda + \frac{k^2}{4}\right|}.$$
Therefore, for $i=1,2,3,4$, we have
\begin{equation}\label{estim alpha i BC3 1}
\left\|e^{(.-a)T_\lambda} \alpha_i \right\|_{L(a,b;X)} \leqslant \frac{C}{1 + \left|\lambda + \frac{k^2}{4}\right|}\, \|f\|_{L^p(a,b;X)},
\end{equation}
and
\begin{equation}\label{estim alpha i BC3 2}
\left\|e^{(b-.)T_\lambda} \alpha_i \right\|_{L(a,b;X)} \leqslant \frac{C}{1 + \left|\lambda + \frac{k^2}{4}\right|}\, \|f\|_{L^p(a,b;X)},
\end{equation}
Finally, from \eqref{estim u1}, \eqref{u3}, \eqref{estim alpha i BC3 1} and \eqref{estim alpha i BC3 2}, we obtain
$$\|u_3\|_{L^p(a,b;X)} \leqslant \frac{C}{1 + \left|\lambda + \frac{k^2}{4}\right|}\, \|f\|_{L^p(a,b;X)}.$$
Now, we focus ourselves on $-\A_4$. Recall that, from \eqref{Repre u bis}, \eqref{PHI 1 2 tilde} and \eqref{ali2 gen}, $u_4$ is given, for all $x \in [a,b]$, by
\begin{equation}\label{u4}
\begin{array}{rcl}
u_4(x)&=&\dis~~ \,\left(e^{(x-a) M_\lambda} - e^{(b-x) M_\lambda}\right)\alpha_1 + \left(e^{(x-a) L_\lambda} - e^{(b-x) L_\lambda}\right)\alpha_2 \\ \ecart
&&\dis+\left(e^{(x-a) M_\lambda} + e^{(b-x) M_\lambda}\right)\alpha_3 + \left(e^{(x-a) L_\lambda} + e^{(b-x) L_\lambda}\right)\alpha_4 + F_{0 ,f}(x),
\end{array}
\end{equation}
where $F_{0,f} = u_1$ is given by \eqref{F0} and
\begin{equation*}
\left\{\begin{array}{rcl}
\alpha_1 & = & \dis - \frac{1}{2} B_\lambda^{-1} (L_\lambda+M_\lambda) V_\lambda^{-1} \left(I - e^{cL_\lambda}\right) L_\lambda  M_\lambda^{-1} \left(F_{0,f}'(a) + F'_{0,f}(b)\right) \\ \ecart

\alpha_2 & = & \dis ~~\frac{1}{2} B_\lambda^{-1} (L_\lambda +M_\lambda) V_\lambda^{-1} \left(I - e^{cM_\lambda}\right) M_\lambda L_\lambda^{-1} \left(F_{0,f}'(a) + F'_{0,f}(b)\right) \\ \ecart

\alpha_3 & = & \dis - \frac{1}{2} B_\lambda^{-1} (L_\lambda+M_\lambda) U_\lambda^{-1} \left(I + e^{cL_\lambda}\right) L_\lambda M_\lambda^{-1} \left(F_{0,f}'(a) - F'_{0,f}(b)\right)  \\ \ecart

\alpha_4 & = & \dis ~~\frac{1}{2} B_\lambda^{-1} (L_\lambda+M_\lambda) U_\lambda^{-1} \left(I + e^{cM_\lambda}\right) M_\lambda L_\lambda^{-1} \left(F_{0,f}'(a) - F'_{0,f}(b)\right).
\end{array}\right.
\end{equation*}
As previously, from \eqref{estim B-1}, \eqref{minoration lambda + k2/4}, \eqref{estim U-1 V-1}, \eqref{estim Z et W}, \refL{Lem LM^-1} and \refL{Lem estim norme F'}, for $i=1,2,3,4$, we have
\begin{equation}\label{estim alpha i BC4 1}
\left\|e^{(.-a)T_\lambda} \alpha_i \right\|_{L(a,b;X)} \leqslant \frac{C}{1 + \left|\lambda + \frac{k^2}{4}\right|}\, \|f\|_{L^p(a,b;X)},
\end{equation}
and
\begin{equation}\label{estim alpha i BC4 2}
\left\|e^{(b-.)T_\lambda} \alpha_i \right\|_{L(a,b;X)} \leqslant \frac{C}{1 + \left|\lambda + \frac{k^2}{4}\right|}\, \|f\|_{L^p(a,b;X)}.
\end{equation}
Finally, from \eqref{estim u1}, \eqref{u4}, \eqref{estim alpha i BC4 1} and \eqref{estim alpha i BC4 2}, we obtain
$$\|u_4\|_{L^p(a,b;X)} \leqslant \frac{C}{1 + \left|\lambda + \frac{k^2}{4}\right|}\, \|f\|_{L^p(a,b;X)}.$$ 
\end{enumerate}
\end{proof}

\subsection{Proof of \refP{Prop Sect}}\label{sect sol Prob P}

For $i=1,2,5$. From \refP{Prop rho(-Ai)}, we have 
$$\sigma\left(-\A_i+\frac{k^2}{4} I\right) \subset \overline{S_{2\theta_A}},$$
and from \refP{Prop estim norme}, for all $\lambda \in \dis \left(\CC \setminus \overline{S_{2\theta_A}}\right)$, there exist $C_i \geqslant 1$, such that
$$\left\|\left(-\A_i + \frac{k^2}{4} I- \lambda I\right)^{-1}\right\|_{\L(X)} \leqslant \frac{C_i}{1 + |\lambda|}.$$
Then, due to \refD{Def op Sect}, we deduce that $\dis -\A_i + \frac{k^2}{4} I$, is a sectorial operator of angle $2\theta_A$.

For $i=3,4$. From \refP{Prop rho(-Ai)2}, there exists $r > 0$, such that  \begin{equation}\label{spectre A3, A4}
\sigma\left(-\A_i + \frac{k^2}{4}I\right) \subset  \overline{B\left(-\frac{k^2}{4},r\right)} \cup \overline{S_{2\theta_A}}.
\end{equation}
Thus
$$\left\{\begin{array}{ll}
\dis \sigma\left(-\A_i + \frac{k^2}{4}I + r I \right) \subset \overline{S_{\frac{\pi}{2}}}, & \text{if }2\theta_A < \frac{\pi}{2} \\ \ecart

\dis \sigma\left(-\A_i + \frac{k^2}{4} I + r I \right) \subset \overline{S_{2\theta_A}}, & \text{if }2\theta_A \geqslant \frac{\pi}{2}.
\end{array}\right.$$
Moreover, due to \refP{Prop estim norme}, there exist constants $C_i \geqslant 1$, such that for all complex numbers $\dis\lambda \in -\frac{k^2}{4} + \left(\CC  \setminus \left(\overline{B\left(0,r\right)} \cup \overline{S_{2\theta_A}}\right)\right)$, we have 
$$\left\|\left(-\A_i + \frac{k^2}{4} I + r I- \lambda I\right)^{-1}\right\|_{\L(X)} \leqslant \frac{C_i}{1 + |\lambda|}.$$
Then, due to \refD{Def op Sect}, we deduce that $\dis -\A_i + \frac{k^2}{4} I + r I$, is a sectorial operator of angle $\dis\frac{\pi}{2}$ if $\dis 2\theta_A < \frac{\pi}{2}$ or of angle $2\theta_A$ if $\dis 2\theta_A \geqslant \frac{\pi}{2}$.

Moreover, when $\dis 2\theta_A \in \left(0,\frac{\pi}{2}\right)$, it is clear from \eqref{spectre A3, A4} that, there exists $r' > r$ large enough, such that 
$$\sigma\left(-\A_i + \frac{k^2}{4} I + r' I \right) \subset \overline{S_{2\theta_A}},$$
and in the same way, if $\theta_A = 0$, there exists $\theta_0 > 0$ such that
$$\sigma\left(-\A_i + \frac{k^2}{4} I + r' I \right) \subset \overline{S_{\theta_0}},$$
which gives the result.

\section{Application} \label{sect Applications}

Let $T>0$ and $i=1,2,3,4,5$. Recall problem \eqref{P parabo}:
\begin{equation}\label{P parabo 2}
\left\{\begin{array}{l}
\dis v'(t) - \A_i v(t) = f(t), \quad t\in (0,T] \\ \ecart
v(0) = v_0,
\end{array}\right.
\end{equation}
which is set in the space 
$$\E := L^p(a,b;X).$$ 
Here $\A_i$ is defined by \eqref{Ai}. We assume that
\begin{center}
\begin{tabular}{l}
$(\H_1)\quad X$ is a UMD space, \\ \ecart
$(\H_2)\quad 0 \in \rho(A)$, \\ \ecart
$(\H_3)\quad -A \in$ BIP$(X,\theta_A),~$ for $\theta_A \in [0, \pi/4)$,\\ \ecart
$(\H_4)\quad [k,+\infty) \in \rho(A)$.
\end{tabular}
\end{center}
Thus, due to \refT{Th final}, $\A_i$ is the infinitesimal generator of a strongly continuous analytic semigroup in $\E$. Note that if $k\geqslant 0$, $(\H_4)$ follows from $(\H_2)$ and $(\H_3)$.

In the sequel, we will consider two cases: 
\begin{enumerate}
\item $f \in W^{\theta,p}(0,T;\E)$, $\theta \in \left(0,\frac{1}{p}\right)$,

\item $f \in C^{\theta}([0,T];\E)$, $\theta \in (0,1)$.
\end{enumerate}

\subsection{First case}

Assume that $f \in W^{\theta,p}(0,T;\E)$ with $\theta \in \left(0, \frac{1}{p}\right)$ (see \cite{da prato-grisvard}, p. 330, for the definition of such a space). From \cite{da prato-grisvard}, Theorem 4.7, p. 334, there exists a unique classical solution $u$ of problem \eqref{P parabo 2} if and only if $v_0 \in \left(D(\A_i),\E\right)_{\frac{1}{p},p}$.

\subsection{Second case}

Now, assume that $f \in C^{\theta}([0,T];\E)$, $\theta \in (0,1)$ and $v_0 \in D(\A_i)$. We then apply Theorem 4.5, p.~53 in \cite{sinestrari} to obtain the existence and the uniqueness of a solution $v \in C([0,T];\E)$ to problem \eqref{P parabo 2} such that
$$v' \in C^\theta([0,T];\E) \quad \text{and} \quad \A_i v(.) \in C^\theta([0,T];\E),$$
if and only if
$$f(0) + \A_i v_0 \in \left(D(\A_i),\E\right)_{1-\theta,+\infty}.$$
Remark that in this case, $\overline{D(\A_i)} = \E$.

\begin{Rem}\hfill
\begin{enumerate}
\item We can prove easily that operators $-\A_1$ and $-\A_5$ are BIP\,$(\E,2\theta_A + \varepsilon)$, $\varepsilon > 0$. Therefore, using the Dore-Venni Theorem, we can solve \eqref{P parabo 2} for $f \in L^p(0,T;\E)$.

\item We can explicit all the above results in the concrete case, that is: 
$$A = A_0 \quad \text{and} \quad \E = L^p(a,b;L^p(\omega)).$$
\end{enumerate}
\end{Rem}

\section*{Acknowledgments} 

\noindent The authors would like to thank the referees for the valuable comments and corrections which help to improve this paper.


\begin{thebibliography}{biblio}

\bibitem{arendt-bu-haase}{\sc W.~Arendt, S.~Bu, M.~Haase}, ``Functional calculus, variational methods and Liapunov's theorem'', {\it Arch. Math.}, 77, 2001, pp. 65-75.

\bibitem{bourgain}{\sc J.~Bourgain}, ``Some remarks on Banach spaces in which martingale difference sequences are unconditional'',  \textit{Ark. Mat.}, vol. 21,1983, pp. 163-168.

\bibitem{burkholder}{\sc D.L.~Burkholder}, ``A geometrical characterisation of Banach spaces in which martingale difference sequences are unconditional'', \textit{Ann. Probab.}, vol. 9, 1981, pp. 997-1011. 

\bibitem{CV 1} {\sc F. Cakoni, G. C. Hsiao \& W. L. Wendland}, ``On the boundary integral equation methodfor a mixed boundary value problem of thebiharmonic equation'', \textit{Complex variables}, Vol. 50, No. 7-11, 2005, pp. 681-696.

\bibitem{cohen-murray} {\sc D.S.~Cohen \& J.D.~Murray}, ``A generalized diffusion model for growth and dispersal in population'', \textit{Journal of Mathematical Biology}, 12, Springer-Verlag, 1981, pp. 237-249.

\bibitem{elasticity 1}{\sc M. Costabel, E. Stephan \& W. L. Wengland}, ``On boundary integral equations of the first kind for the bi-Laplacian in a polygonal plane domain'', \textit{Annali della Scuola Normale Superiore di Pisa, Classe di Scienze}, $4^{\text{e}}$ serie, tome  10, no 2, 1983, pp. 197-241.

\bibitem{da prato-grisvard} {\sc G. Da Prato \& P. Grisvard}, ``Sommes d'op\'erateurs lin\'eaires et \'equations diff\'erentielles op\'erationnelles'', \textit{J. Math. Pures et Appl.}, 54, 1975, pp. 305-387.

\bibitem{dore} G. Dore, \textit{Lp Regularity for Abstract Differential Equation}s, in H. Komatsu (Ed.), "Functional Analysis and Related Topics, 1991 (Proceedings, Kyoto 1991)", Lecture Notes in Math. 1540, Springer, Berlin Heidelberg New York (1993), pp. 25-38.

\bibitem{dore-venni} {\sc G. Dore \& A. Venni}, ``On the closedness of the sum of two closed operators'', \textit{Math. Z.}, 196, 1987, pp. 189-201.

\bibitem{dore-yakubov} {\sc G. Dore \& S. Yakubov}, ``Semigroup estimates and noncoercive boundary value problems'', \textit{Semigroup Forum}, Volume 60, Issue 1, pp. 93-121.

\bibitem{engel-nagel} {\sc K.-J.~Engel \& R.~Nagel}, \textit{One-Parameter Semigroups for Linear Evolution Equations}, Springer, 2000.

\bibitem{hassan}{\sc A.~Favini, R.~Labbas, K.~Lemrabet, S.~Maingot \& H.~Sidib\'e}, ``Transmission Problem for an Abstract Fourth-order Differential Equation of Elliptic Type in UMD Spaces'', \textit{Advances in Differential Equations}, Volume 15, Numbers 1-2, pp. 43-72, 2010.

\bibitem{falamaintaya} {\sc A.~Favini, R.~Labbas, S.~Maingot, H.~Tanabe \& A.~Yagi}, ``A simplified approach in the study of elliptic differential equations in UMD spaces and new applications'', {\it Funkc. Ekv.}, 451, pp. 165-187, 2008.

\bibitem{FLMT} {\sc A.~Favini, R.~Labbas, S.~Maingot \& A. Thorel}, ``Elliptic differential-operator with an abstract Robin boundary condition containing two spectral parameters, study in a non commutative framework'', \textit{submitted}, 2020.  \url{hal-02975665}

\bibitem{plaque 1}{\sc G. Geymonat \& F. Krasucki}, ``Analyse asymptotique du comportement en flexion de deux plaques coll\'ees'', \textit{C. R. Acad. Sci. Paris - II b}, Vol 325, Issue 6, 1997, pp. 307-314.

\bibitem{grisvard} {\sc P. Grisvard}, ``Spazi di tracce e applicazioni'', \textit{Rendiconti di Matematica}, (4) Vol.5, Serie VI, 1972, pp. 657-729.

\bibitem{electrostatic 1}{\sc Z. Guo, B. Lai \& D. Ye}, ``Revisiting the biharmonic equation modelling electrostatic actuation in lower dimensions'', \textit{Proc. Amer. Math. Soc.}, 142, 2014, pp. 2027-2034. 

\bibitem{haase} {\sc M. Haase}, \textit{The functional calculus for sectorial Operators}, Birkhauser, 2006.

\bibitem{elasticity 2}{\sc M. A. Jaswon \& G. T. Symm}, ``Integral equation methods in potential theory and elastostatics'', \textit{Academic Press}, New York, San Francisco, London, 1977, pp. 1-10.

\bibitem{komatsu} {\sc H. Komatsu}, ``Fractional powers of operators'', \textit{Pacific Journal of Mathematics}, Vol. 19, No. 2, 1966, pp. 285-346.

\bibitem{LMMT} {\sc R. Labbas, S. Maingot, D. Manceau \& A. Thorel}, ``On the regularity of a generalized diffusion problem arising in population dynamics set in a cylindrical domain'', \textit{Journal of Mathematical Analysis and Applications}, 450, 2017, pp. 351-376.

\bibitem{LLMT} {\sc R.~Labbas, K.~Lemrabet, S.~Maingot \& A.~Thorel}, ``Generalized linear models for population dynamics in two juxtaposed habitats'', \textit{Discrete and Continuous Dynamical Systems - A}, Volume 39, Number 5, May 2019, pp. 2933-2960. 

\bibitem{electrostatic 2}{\sc F. Lin \& Y. Yang},``Nonlinear non-local elliptic equation modelling electrostatic actuation'', \textit{Proceedings of the Royal Society of London A}, 463, 2007, pp. 1323-1337.

\bibitem{lions-peetre}  {\sc J.-L.~Lions \& J.~Peetre}, ``Sur une classe d'espaces d'interpolation'', \textit{Publications math\'ematiques de l'I.H.\'E.S.}, tome 19, 1964, pp. 5-68.

\bibitem{lunardi}{\sc A.~Lunardi}, \textit{Analytic semigroups and optimal regularity in parabolic problems}, Bir-khäuser, Basel, Boston, Berlin, 1995.

\bibitem{lunardi 2}{\sc A.~Lunardi}; \textit{Interpolation theory}, Third edition, Lecture Notes, Scuola Normale Superiore di Pisa (New Series), 16. Edizioni della Normale, Pisa, 2018. 

\bibitem{monniaux} {\sc S.~Monniaux}, ``A perturbation result for bounded imaginary powers'', \textit{Archiv der Mathematik}, 68, 1997, p. 407-417.

\bibitem{ochoa} {\sc F.L. Ochoa}, ``A generalized reaction-diffusion model for spatial structures formed by motile cells'', \textit{BioSystems}, 17, 1984, pp. 35-50.

\bibitem{okubo}{\sc A. Okubo \& S.A. Levin}, \textit{Diffusion and ecological problems, Mahematical biology}, second edition, Springer-Verlag, Berlin, 2010.

\bibitem{pruss} {\sc J.~Pr\"uss}, \textit{Evolutionary Integral Equations and Applications}, Birkhäuser Verlag, Switzerland, 1993.

\bibitem{pruss-sohr} {\sc J.~Pr\"uss \& H.~Sohr}, ``On operators with bounded imaginary powers in Banach spaces'', \textit{Mathematische Zeitschrift, Springer-Verlag, Math. Z.}, 203, 1990, pp. 429-452.

\bibitem{IJPAM} {\sc H. Saker \& N. Bouselsal}, ``On the bilaplacian problem with nonlinear boundary conditions'', \textit{Indian J. Pure Appl. Math.}, Volume 47, Issue 3, 2016, pp. 425–435.

\bibitem{sinestrari} {\sc E. Sinestrari}, ``On the abstract Cauchy problem of parabolic type in spaces of continuous functions'', \textit{Journal of Mathematical Analysis and Applications}, 107, 1985, pp. 16-66. 

\bibitem{tanabe}{\sc H.~Tanabe}, \textit{Equations of evolution}, Pitman Publishing Ltd, 1979.

\bibitem{thorel} {\sc A.~Thorel}, ``Operational approach for biharmonic equations in $L^p$-spaces'', \textit{Journal of Evolution Equations}, 20, 2020, pp. 631-657.

\bibitem{plaque 2}{\sc I. Titeux \& E. Sanchez-Palencia}, ``Conditions de transmission pour les jonctions de plaques minces'', \textit{C. R. Acad. Sci. Paris - serie II b}, Volume 325, Issue 10, 1997, pp. 563-570.

\bibitem{triebel} {\sc H.~Triebel}, \textit{Interpolation theory, function Spaces, differential Operators}, North-Holland publishing company Amsterdam New York Oxford, 1978.

\end{thebibliography}
\end{document}